\documentclass[12pt]{amsart}
\usepackage{amsmath}
\usepackage{amsfonts,amssymb}
\usepackage{euscript}
\usepackage{amsthm}
\usepackage[matrix,arrow,curve]{xy}
\numberwithin{equation}{section}
\theoremstyle{plain}
\newtheorem{theorem}{Theorem}[section]
\newtheorem{theoremintro}{Theorem}
\newtheorem{lemma}[theorem]{Lemma}
\newtheorem{prop}[theorem]{Proposition}
\newtheorem{corollary}[theorem]{Corollary}

\theoremstyle{definition}
\newtheorem{definition}[theorem]{Definition}
\newtheorem{remark}[theorem]{Remark}
\newtheorem{example}[theorem]{Example}

\newcommand{\e}{\varepsilon}

\newcommand{\p}{\mathfrak p}

\newcommand{\Z}{\mathbb Z}

\renewcommand{\P}{\mathbb P}

\renewcommand{\SS}{\mathcal S}

\newcommand{\ZZ}{\mathcal Z}

\newcommand{\GG}{\mathcal G}

\newcommand{\LL}{\mathcal L}

\newcommand{\TT}{\mathcal T}
\newcommand{\PP}{\mathcal P}
\newcommand{\QQ}{\mathcal Q}
\newcommand{\XX}{\mathcal X}
\newcommand{\YY}{\mathcal Y}

\renewcommand{\k}{\mathsf k}

\newcommand{\mmod}{\mathrm{{-}mod}}
\newcommand{\modd}{\mathrm{mod{-}}}
\newcommand{\grmod}{\mathrm{grmod{-}}}

\newcommand{\xra}{\xrightarrow}
\renewcommand{\le}{\leqslant}
\renewcommand{\ge}{\geqslant}
\renewcommand{\~}{\widetilde}
\newcommand{\bul}{\bullet}

\DeclareMathOperator{\Hom}{\textup{Hom}}

\DeclareMathOperator{\Ext}{\textup{Ext}}

\DeclareMathOperator{\End}{\mathrm{End}}

\DeclareMathOperator{\Spec}{\mathrm{Spec}}

\DeclareMathOperator{\coker}{\mathrm{coker}} 

\DeclareMathOperator{\im}{\mathrm{im}}

\DeclareMathOperator{\rank}{\mathrm{rank}}

\DeclareMathOperator{\Supp}{\mathrm{Supp}}
\DeclareMathOperator{\coh}{\mathrm{coh}}

\DeclareMathOperator{\Perf}{\mathrm{Perf}}

\DeclareMathOperator{\op}{\mathrm{op}}
\DeclareMathOperator{\Proj}{\mathrm{Proj}}

\DeclareMathOperator{\pd}{\mathrm{pd}}
\DeclareMathOperator{\codim}{\mathrm{codim}}
\DeclareMathOperator{\Mat*}{\Hom_{\mathrm{ff}}}

\author{Alexey Elagin}
\address{Institute for Information Transmission Problems (Kharkevich Institute), Moscow, RUSSIA\\
National Research University Higher School of Economics, Moscow, RUSSIA}
\email{alexelagin@rambler.ru}

\author{Valery A.~Lunts}
\address{
  Department of Mathematics\\
  Indiana University\\
  Rawles Hall\\
  831 East 3rd Street\\
  Bloomington, IN 47405\\
  USA
}
\email{vlunts@indiana.edu}

\title[Derived categories of zero-dimensional schemes]
{Derived categories of coherent sheaves on some zero-dimensional schemes}

\begin{document}

\begin{abstract}
Let $X_N$ be the second infinitesimal neighborhood of a closed point in $N$-dimensional affine space. In this note we study $D^b(\coh X_N)$, the bounded derived category of coherent sheaves on $X_N$.
We show that for $N\geq 2$ the lattice of triangulated subcategories in $D^b(\coh X_N)$ has a rich structure (which is probably wild), in contrast to the case of zero-dimensional complete intersections.  Our homological methods produce some applications to universal localizations of free associative algebras. These applications are based on  a relation between  triangulated subcategories in $D^b(\coh X_N)$ and universal localizations of a free graded associative algebra in $N$ variables.
\end{abstract}

\maketitle


\section{Introduction}

In this note we study bounded derived categories of coherent sheaves on affine Noetherian schemes. Compared to more popular derived categories of projective schemes, they demonstrate quite different properties. For example, $\Hom$ spaces in derived categories of affine schemes are usually not finite-dimensional/of finite type, these derived categories do not admit nontrivial semi-orthogonal decompositions or interesting autoequivalences.

More concretely, we are interested in thick triangulated subcategories in derived categories of affine schemes. Can these be classified? The starting point here is a remarkable theorem by Hopkins and Neeman \cite{HN}:

\begin{theorem}
\label{theorem_Hopkins}
Let $R$ be a Noetherian ring and $X=\Spec R$. Then there is a bijection
\begin{align*}
\left\{\text{thick triangulated subcategories in}\,\,\Perf X \right\}
&\longleftrightarrow
\left\{\text{subsets in}\,\, X,\,\,\text{closed under specialization} \right\},\\
T & \mapsto Z_T:=\cup_{F\in T}\Supp F,\\
T_Z:=\{F\in \Perf X\mid \Supp F\subset Z\} & \gets Z.
\end{align*}
The above bijection restricts to a bijection between thick triangulated finitely generated subcategories in $\Perf X$ and closed subsets in $X$.
\end{theorem}

For general affine scheme $X$ there is no classification of thick triangulated subcategories in $D^b(\coh X)$. Of course, for affine schemes $X$ such that $\Perf X=D^b(\coh X)$ (for example, $X$ can be a regular affine scheme of finite Krull dimension) all thick triangulated subcategories in $D^b(\coh X)$ are also described by Theorem~\ref{theorem_Hopkins}. Also, for some  schemes $X$ there is a classification of
thick triangulated subcategories in the category of singularities
$$D_{sg}(X):=D^b(\coh X)/\Perf X.$$
This is equivalent to classifying thick triangulated subcategories in $D^b(\coh X)$ containing $\Perf X$.
For example, one has the following result by R.\,Takahashi and G.\,Stevenson for affine hypersurfaces.
\begin{theorem}[See {\cite[Th. 6.13]{St}}]
\label{th_lochyper}
Let $R$ be a Noetherian ring which is locally a hypersurface (i.e., any localization of $R$ at a prime ideal is isomorphic to a quotient by a principal ideal of a regular Noetherian ring), let $X=\Spec R$. Then there is an order preserving bijection between thick triangulated subcategories in $D_{sg}(X)$ and specialization-closed subsets in $\mathrm{Sing}\, X$.
\end{theorem}

There is a generalization of Theorem~\ref{th_lochyper} to local complete intersections, see  \cite[Theorem 8.8 and Corollary 10.5]{St}.

\medskip
Consider the ``extremal case'': let $X$ be a zero-dimensional connected scheme. Then  $\Perf X$ has only two thick triangulated subcategories: $0$ and $\Perf X$ by Theorem~\ref{theorem_Hopkins}. Still  $D^b(\coh X) $ can have quite many ones. In some cases they are classified:

\begin{theorem}[See {\cite[Th. 5.6, Remark 5.12]{CI}}]
\label{theorem_CI} Let $\k$ be a field and let $R$ be a complete intersection ring of dimension zero:
$$R=\k[z_1,\ldots,z_r]/(f_1,\ldots,f_r)$$
where $f_1,\ldots,f_r$ is a regular sequence in $(z_1,\ldots,z_r)^2$.
Let  $\k[\theta_1,\ldots,\theta_r]$ be the graded polynomial algebra with $\deg \theta_i=1$. Denote by $\Spec^*\k[\theta_1,\ldots,\theta_r]$ the set of prime homogeneous ideals in  $\k[\theta_1,\ldots,\theta_r]$ with Zariski topology. Then there is a bijection between thick triangulated subcategories in $D^b(R\mmod)$ and specialization-closed subsets in $\Spec^*\k[\theta_1,\ldots,\theta_r]$. This bijection preserves inclusion of subcategories/subsets. Under this bijection finitely generated subcategories correspond to closed subsets.
\end{theorem}

In the cited results the key instrument the establish the bijections is the notion of \emph{support} of an object in $D^b(R\mmod)$. This support is a closed subset in some reasonable Noetherian topological space (like $\mathrm{Sing}\, X$ in Theorem~\ref{th_lochyper} or $\Spec^*\k[\theta_1,\ldots,\theta_r]$ in Theorem~\ref{theorem_CI}). Consequently, the lattice of
thick triangulated subcategories in $D^b(R\mmod)$ demonstrates tame behavior, as well as
the lattice of specialization-closed subsets in the corresponding topological space does. For example, we have a direct corollary of Theorem~\ref{theorem_CI}.
\begin{prop}[See Proposition~\ref{prop_lattice1}]
Let $R$ be a complete intersection ring of dimension zero as in Theorem \ref{theorem_CI}. Then the lattice~$\LL(R)$ of nonzero finitely generated thick triangulated subcategories in $D^b(R\mmod)$ has the following properties:
\begin{enumerate}
\item $\LL(R)$ satisfies the descending chain condition;
\item there exists the least element $\Perf R$ in $\LL(R)$, let us call minimal elements in $\LL(R)\setminus \{\Perf R\}$ \emph{almost minimal};
\item let $T_1,T_2\in \LL(R)$ be almost minimal. Then there exists finitely many (in fact, only four) elements in $\LL(R)$ bounded above by $\langle T_1, T_2\rangle$ (that is, only four subcategories in $\langle T_1,T_2\rangle$): $\Perf R, T_1,T_2,\langle T_1,T_2\rangle$.
\end{enumerate}
\end{prop}

It is believed that in general the structure of the lattice of thick triangulated subcategories in $D^b(\coh X)$ is wild {\cite{Ne}}.

In this paper we deal with a simpliest zero-dimensional scheme which is not a complete intersection. For a fixed field $\k$ we consider algebras
$$R_N:=\k[y_1,\ldots,y_N]/(y_1,\ldots,y_N)^2$$
and schemes $X_N=\Spec R_N$ for $N\geq 2$. We establish some properties of thick triangulated subcategories in $D^b(\coh X_N)$ supporting the belief that the classification of such subcategories is a wild problem.

Our main results in this direction are the following.

We show that the lattice of thick triangulated subcategories in $D^b(R\mmod)$ has quite different behavior for $R=R_N$ than for complete intersections from Theorem~\ref{theorem_CI}. We have

\begin{theoremintro}[See Proposition~\ref{prop_lattice2}]
\label{theorem_lattice2intro}
Let  $R_N=\k[y_1,\ldots,y_N]/(y_1,\ldots,y_N)^2$, $N\ge 2$.
The lattice $\LL(R_N)$ of nonzero finitely generated thick triangulated subcategories in $D^b(R_N\mmod)$ has the following properties:
\begin{enumerate}
\item $\LL(R_N)$ does not satisfy the descending chain condition.
\item There exists no least element in $\LL(R_N)$. Category $\Perf R_N$ is minimal and almost maximal in $\LL(R_N)$: there are no elements between $\Perf R_N$ and $D^b(R_N\mmod)$.
\item There exist minimal elements $T_1,T_2\in \LL(R_N)$ such that there are infinitely many elements bounded above by $\langle T_1,T_2\rangle$ .
\end{enumerate}
\end{theoremintro}

A following strengthening of property (1) from Theorem~\ref{theorem_lattice2intro} is possible.
\begin{theoremintro}[See Theorem \ref{theorem_tree}]
Let  $R_N=\k[y_1,\ldots,y_N]/(y_1,\ldots,y_N)^2$, $N\ge 2$.
Then there exists an infinite descending binary tree of embedded
finitely generated thick triangulated subcategories in $D^b(R_N\mmod)$.
\end{theoremintro}

\medskip
As a main tool, we use the equivalence of triangulated categories
$$D^b(R_N\mmod)\to \Perf A_N,$$
where $A_N$ denotes the free graded algebra in (non-commuting) variables $x_1,\ldots,x_N$ of degree $1$.
Here we consider $A_N$ as a dg $\k$-algebra with zero differential and $\Perf A_N$ is the triangulated category of perfect dg $A_N$-modules. This algebra $A_N$ is hereditary.

To construct examples of subcategories in $D^b(R_N\mmod)$, we use  graded $A_N$-modules of a special kind. For any homomorphism $g$ of free finitely generated graded $A_N$-modules we denote by $M_g$ the cone of $g$, it is a dg module over $A_N$.
In particular, for a homogeneous  element  $x\in A_N$ we denote by $M_x$ the cone
$Cone(A_N[-\deg x]\xra{x} A_N)$.  Moreover, any object in $\Perf A_N$ is isomorphic to  $M_g$ for some homomorphism~$g$ as above.  For a family of homogeneous elements $\XX$ in $A_N$ we introduce the notion of $\XX$-filtration   on a graded  $A_N$-module. Using this notion we determine when a family of dg $A_N$-modules $(M_g, g\in \XX)$ generates (in $\Perf A_N$) a given dg $A_N$-module $M_x$ for a homogeneous $x\in A_N$.
Under an important  technical condition on $\XX$ called \emph{goodness} we prove that $M_x\in \langle M_g\rangle_{g\in \XX}$ if and only if the graded $A_N$-module $A_N/xA_N$ has an $\XX$-filtration (Lemma~\ref{lemma_Gdiag}) and if and only if~$x$ is a product of a scalar and several elements from $\XX$ (see Theorems~\ref{theorem_xprod} and~\ref{theorem_xprod2}).

\medskip

We provide another criterion of generation for modules of the form $M_g$, in terms of universal localizations of a free graded algebra in sense of P.\,Cohn.

\begin{theoremintro}[See Corollary {\ref{cor_G1G2}}]
\label{theoremintro_cor}
Let $x$ be a homomorphism of finitely generated free graded  $A_N$-modules and let $\SS$ be a family of such homomorphisms. Then $M_x\in \langle M_g\rangle_{g\in \SS}$ if and only if the homomorphism $x$ is invertible over $A[\SS^{-1}]$.
\end{theoremintro}

This criterion is based on a relation between thick triangulated subcategories in $\Perf A_N$ and universal localizations of the free graded algebra $A_N$. 
For any family $\SS$ of homomorphisms of finitely generated free graded $A_N$-modules there is a well-defined homomorphism of graded algebras $\lambda_\SS\colon A_N\to  A_N[\SS^{-1}]$, called universal localization. In 
Proposition~\ref{equality of localizations} we explain that, roughly speaking, the scalar extension functor $\Perf A_N\to \Perf A_N[\SS^{-1}]$ is isomorphic to 
the Verdier localization functor 
$$\Perf A_N\to (\Perf A_N)/\langle M_g\rangle_{g\in \SS}.$$
Such results are not new, see for example, \cite{NR}, \cite{CY} or \cite{KY}. After our preprint was published we have learned that a statement equivalent to ours is contained in \cite[Prop. 3.5]{CY}. Nevertheless we decided to leave our proof (following ideas of \cite{NR}) in the paper, as it is quite different from the one given in \cite{CY}.

In view of Theorem~\ref{theoremintro_cor}, questions about subcategories in $D^b(R_N\mmod)$ or $\Perf A_N$ are equivalent to questions about universal localizations of a free algebra.
Once we can solve some questions on the categorical side, we get some consequences for localizations of algebras. We do not know if these results can be obtained directly.

\begin{theoremintro}[See Theorem {\ref{theorem_invert}} and Corollary {\ref{cor_locinjective}}]
Let $A$ be a free graded algebra. Let $\XX\subset A$ be a family of nonzero homogeneous elements. Then there exists a family  $\~\XX\subset A$ of nonzero homogeneous elements (see Proposition~\ref{prop_makegood} for its construction) such that  a homogeneous element $x\in A$ is invertible in
$A[\XX^{-1}]$ if and only if
$$x=\lambda\cdot \prod_{i=1}^n y_i\quad\text{for some}\quad \lambda\in\k^*, \quad y_i\in\~\XX \quad \text{and}\quad n\ge 0.$$

Also,  the localization $A[\XX^{-1}]$ is not a zero ring and  the canonical map
$A\to A[\XX^{-1}]$
is injective.
\end{theoremintro}

The authors are grateful to Amnon Neeman for valuable discussions and to Xiao-Wu Chen for making us aware of his paper \cite{CY}.

\subsection{Some definitions, conventions and notation}

In this paper we work over a fixed field $\k$. All rings are supposed to be  associative and unital, by algebras we always mean associative  unital $\k$-algebras. All graded algebras and rings are $\Z$-graded. A {\it free graded algebra} always means a graded algebra which is freely generated by finitely many non-commuting elements of degree~$1$.
All modules are right modules.

Let $\Lambda $ be a graded algebra.
Denote by $\Mat*(\Lambda )$  the set of degree preserving homomorphisms between finitely generated free graded $\Lambda $-modules.
For (arbitrary) graded $\Lambda $-modules $M,N$ we denote by $\Hom_{\grmod \Lambda}(M,N)$ the set of degree preserving homomorphisms $M\to N$ of $\Lambda$-modules. We denote 
$$\Hom^{\bul}_{\grmod \Lambda}(M,N):=\oplus_{i\in\Z}\Hom_{\grmod \Lambda}(M,N[i]),$$
it is a graded vector space. 
We let $\Lambda^*$ be the set of nonzero homogeneous elements in~$\Lambda$.

We consider $\Lambda$ as a dg algebra with zero differential, and denote by $D(\Lambda)$ the corresponding derived category of (right) dg $\Lambda$-modules. As usual $\Perf \Lambda \subset D(\Lambda)$ is the full subcategory of perfect dg modules. Recall that $D(\Lambda)$ and $\Perf \Lambda$ are triangulated categories.
For a dg $\Lambda$-module $M$ we denote by $H(M)$ its cohomology which is a graded $\Lambda$-module.

A full strict subcategory $T'\subset T$ in an additive category is \emph{thick} if it is closed under direct summands that exist in $T$: for any objects $F,F'\in T$ such that $F\oplus F'\in T'$ one has $F,F'\in T'$.

For a triangulated category $T$ and a family of objects $\GG\subset T$, define full subcategories $[\GG]_n\subset T$ as follows. Let $[\GG]_0$ consist of finite direct sums of shifts of objects of $\GG$. Let $[\GG]_n$ consist of objects $M\in T$ such that there exists a triangle $M_{n-1}\to M\to M_0\to M_{n-1}[1]$ with $M_0\in [\GG]_0, M_{n-1}\in [\GG]_{n-1}$.
Let $[\GG]:=\cup_i [\GG]_i$, it is the smallest full triangulated subcategory in $T$ containing $\GG$. Denote by
$\langle \GG\rangle\subset T$ the thick closure of $[\GG]$, it is the smallest full thick triangulated subcategory in $T$ containing $\GG$. It is said to be classically generated by $\GG$.

For a triangulated category $T$ and  objects $E,F\in T$ we put
$$\Hom^{\bul}_T(E,F):=\oplus_{i\in\Z}\Hom_T(E,F[i]),$$
it is a graded abelian group. In particular,
$$\End^{\bul}_T(F):=\oplus_{i\in\Z}\Hom_T(F,F[i])$$
 is a graded ring.

For a commutative Noetherian algebra $R$ we consider the bounded derived category $D^b(\modd R)$ of finitely generated $R$-modules, and its full subcategory $\Perf R\subset D^b(\modd R)$ of perfect $R$-complexes.

\section{The schemes we study}
\label{section_schemes}
Let $\k$ be a fixed field. Let $N\ge 2$ and put $$R_N:=\k[y_1,\ldots,y_N]/(y_1,\ldots,y_N)^2,$$
this is a commutative local finite-dimensional algebra. Let
$$X_N:=\Spec R_N.$$
Sometimes we will omit the index $N$ and write just $R$ and $X$.

In this note we study the bounded derived category $D^b(\modd R)=D^b(\coh X)$. In particular, we are interested in thick  triangulated subcategories in $D^b(\modd R)$.

It follows from Theorem \ref{theorem_Hopkins} that $\Perf R$ has no non-trivial thick triangulated subcategories.

Define dualization functor $*\colon \modd R\to (\modd R)^{\op}$ by taking the dual $\k$-vector space.
We have for a coherent sheaf $F$ on $X$
$$F^*=\Hom_{\k}(F,\k)=\Hom_R(F,R^*).$$
This functor is exact, its derived functor is the Grothendieck duality
$$D^b(\modd R)\xra{\sim} D^b(\modd R)^{\op}$$
coming from the dualizing sheaf $R^*$.

Let $\k$ denote the simple $R$-module. Then $\k$ is a classical generator of $D^b(\modd R)$, one has $\langle \k\rangle_1=D^b(\modd R)$ (and moreover $[\k]_1=D^b(\modd R)$).
By Koszul duality, the graded algebra 
$\End_{D^b(\modd R_N)}^{\bul}(\k)=\oplus_{i\ge 0}\Ext^i_{R_N}(\k,\k)$ is isomorphic to the free associative algebra $\k\{x_1,\ldots,x_N\}$ in $N$ variables of degree one.
To be more accurate, generators $x_i$ of this algebra are dual to the original $y_i$-s. 
Denote this free associative graded algebra by $A_N$ (or just by $A$).

Choose any dg enhancement of $D^b(\modd R)$ (which is unique by Theorem 8.13 in \cite{LO}). Then the dg algebra $R\End_R(\k)$ is quasi-isomorphic to its graded cohomology algebra $A$, because the latter is free. It follows that one has an equivalence of triangulated categories
$$K\colon D^b(\modd R)\xra{\sim} \Perf A,$$
where $A$ is treated as a dg algebra with zero differential.
It would be more convenient for us to use another equivalence

We will consider the equivalence
$$K'\colon D^b(\modd R)\xra{\sim} (\Perf A)^{\op},$$
which is the composition of the dualization on $X$ and $K$.
One has
$$K'(\k)\cong A,\quad K' (R)\cong \Hom(\k,R^*)\cong \k.$$
It follows that $K'$ restricts to an equivalence
$$\Perf R\xra{\sim} D_{fd}(A)^{\op},$$
where $D_{fd}(A)\subset D(A)$ denotes the full subcategory formed by dg modules with finite-dimensional (over $\k$) cohomology.

\section{Modules over free algebras}
Let $A=A_N=\k\{x_1,\ldots,x_N\}$ be the free graded algebra in $N$ variables, $\deg x_i=1$.
The algebra $A$  has global dimension $1$. Moreover $A$ is a free ideal ring: any right or left ideal in $A$ is free as an $A$-module.
In particular, $A$ has no zero divisors.
The algebra~$A$ is not Noetherian for $N\ge 2$, but is right graded coherent: any finitely generated homogeneous right ideal in $A$ is finitely presented. Consequently, the category of finitely presented graded $A$-modules is abelian.

We have a standard
\begin{lemma}[See {\cite[Prop. 1.2.1]{Co}}]
\label{lemma_fir}
Let $A$ be a free graded  algebra, let $M\subset N$ be graded $A$-modules. If $N$ is free then $M$ is also free.
\end{lemma}

The following Proposition is well-known in the non-graded setting, but we prefer to give a careful proof for the graded case. Following our convention we consider $A$ as a dg algebra with zero differential, $D(A)$ is the derived category of dg $A$-modules.
\begin{prop}
\label{prop_module} For the free graded algebra $A$ we have:
\begin{enumerate}
\item Any object in $M\in D(A)$ is isomorphic to
the cone of a homomorphism $g\colon F_1\to F_0$, where $F_0$ and $F_1$ are
free graded $A$-modules and $g\in \Hom_{\grmod A}(F_1,F_0)$ is an injective  homomorphism. In particular, any dg module $M\in D(A)$ is quasi-isomorphic to its cohomology graded module $H(M)$.

\item In the above notation, if $M\in \Perf A$ then $F_0,F_1$ can be chosen to be finitely generated free graded $A$-modules.

\item A dg $A$-module $M$ is in $\Perf A$ if and only if $H(M)$ is a finitely presented $A$-module.
\end{enumerate}
\end{prop}
\begin{proof}
(1) Let $M$ be a graded dg $A$-module with the differential $d$. Choose a surjection $s\colon F_0\to H(M)$ for some free graded $A$-module $F_0$, let $F_1$ be its kernel and $g\colon F_1\to F_0$ be the inclusion.  By Lemma~\ref{lemma_fir}, $F_1$ is also a free graded $A$-module. Since $F_0$ is free, $s$~lifts to a homomorphism $f_0\colon F_0\to Z(M)\subset M$.  Similarly, $f_0g\colon F_1\to M$ lands in the image of $d$ and thus lifts to a homomorphism  $f_1\colon F_1\to M$ such that $f_0g=df_1$. Now $f_0,f_1$ determine a homomorphism of dg $A$-modules $Cone(F_1\xra{g} F_0)\to M$ which is a quasi-isomorphism.

Since $g$ is injective, $Cone(g)$ is quasi-isomorphic to $\coker g=H(M)$ and the second statement follows.

(3) Assume $H(M)$ is finitely presented, then in the proof of (1) $F_0$ and $F_1$  can be taken to be free finitely generated. Thus $F_0,F_1\in \Perf A$ and consequently $M\in\Perf A$.

Now assume $M\in\Perf A$, we need to show that $H(M)$ is finitely presented. Let $\TT\subset D(A)$ denote the full subcategory of dg $A$-modules with finitely presented cohomology. We have $A\in \TT$. Also, $\TT$ is triangulated. Indeed, since algebra $A$ is graded coherent, the full subcategory of finitely presented graded $A$-modules is an abelian and extension-closed subcategory in the abelian category of all graded $A$-modules. Therefore is closed under taking cones and thus triangulated. Also, $\TT$ is thick. It follows that $\Perf A=\langle A\rangle\subset \TT$ and $H(M)$ is finitely presented.

(2) It is clear since $H(M)$ is finitely presented by (3).
\end{proof}

\begin{remark}
Any homomorphism
$$g\colon F_1=\oplus_{j=1}^n \Lambda [c_j]\to \oplus_{i=1}^m \Lambda [d_i] = F_0\in \Mat*(\Lambda )$$
is given by a matrix $G\in Mat_{m\times n}(\Lambda )$ with $G_{ij}\in \Lambda _{d_i-c_j}$. Note that
not any matrix over~$\Lambda $ with homogeneous components defines a homomorphism of free modules (for example, the matrix $\begin{pmatrix} 1& 1\\ 1& x\end{pmatrix}$ does not if $\deg x>0$). Also note that the matrix $G$ does not determine the modules $F_1$ and $F_0$ uniquely (or uniquely up to some shift). Indeed, for $G=0\in Mat_{m\times n}(\Lambda )$ one cannot say anything about grading of $F_1$ and $F_0$.
\end{remark}

\begin{definition}\label{m-g}
For a homomorphism $g\colon F_1\to F_0$ in $\Mat*(\Lambda )$ given by a  matrix $G$  we denote by $M_g$
the cone of $g$, it is a dg $\Lambda $-module. Sometimes we also denote this cone by $M_G$.
In particular, if $x\in \Lambda $ is a homogeneous element of degree $d$ then we denote by $M_x$ the cone $Cone(\Lambda [-d]\xra{x}\Lambda )$.
\end{definition}

For future reference we reformulate part of Proposition~\ref{prop_module} as
\begin{corollary}
\label{cor_module}
Let $A$ be a free graded algebra.
Any object in $\Perf A$ is isomorphic to a graded module $M_g$ for some injective $g\in \Mat*(A)$.
\end{corollary}


\begin{lemma}
\label{lemma_dissection}
Let $A=A_N$ be  a free graded algebra and  $M$ be a finitely presented graded $A$-module.
Then there exists a free submodule $F\subset M$ of finite rank such that $M/F$ is a finite-dimensional $A$-module.
\end{lemma}
\begin{proof}
Let $d\in\Z$. For a graded  $A$-module $L$ denote by $\tau_{\ge d}L$ the graded $A$-module defined as follows
$$(\tau_{\ge d}L)_i:=L_i\quad\text{for}\quad i\ge d,\quad \text{and}\quad (\tau_{\ge d}L)_i:=0\quad\text{otherwise.}$$
Clearly, $\tau_{\ge d}L$ is a submodule of $L$ and the quotient $L/\tau_{\ge d}L$ is finite-dimensional as soon as $L$ is finitely generated.
Also note that $\tau_{\ge d}$ defines an exact functor on the abelian category of graded $A$-modules.

The reader is welcome to check that
\begin{equation}
\tau_{\ge d}(A[m])\cong \begin{cases} A^{N^{d+m}}[-d] & d+m\ge 0\\ A[m] & d+m<0.\end{cases}
\end{equation}

By Proposition~\ref{prop_module}, $M$ is the cokernel of an injective homomorphism
$g\colon F_1\to F_0$ of free graded  $A$-modules of finite rank.
Now let us take $d$ big enough such that $d+m\ge 0$ for any direct summand $A[m]$ in $F_0$ or in $F_1$. Then we have $\tau_{\ge d}(F_0)\cong A^{a_0}[-d]$ and $\tau_{\ge d}(F_1)\cong A^{a_1}[-d]$ for some $a_0,a_1\in\Z$. The embedding
$$\tau_{\ge d}g\colon \tau_{\ge d}(F_1)=A^{a_1}[-d]\to A^{a_0}[-d]=\tau_{\ge d}(F_0)$$ is given by a rectangular matrix over $\k$, hence it is a split embedding. Therefore $F:=\tau_{\ge d} M\cong \coker(\tau_{\ge d}g)$ is isomorphic to $A^{a_0-a_1}[-d]$, it is a free graded $A$-module of finite rank.

It remains to observe that $M/\tau_{\ge d}M$ is finite-dimensional as noted before.
\end{proof}

%
%

\section{Modules of the form $M_x$}
\label{section_Mx}

In this section we concentrate on subcategories in $\Perf A$ generated by one or several dg $A$-modules $M_x$, where $x\in A$ is a homogeneous element.
Recall that
$$M_x=Cone (A[-d]\xra{x} A),$$
 where  $d=\deg x$. Since $A$ has no zero divisors, $M_x$ is quasi-isomorphic to the cyclic graded $A$-module $A/xA$.

In this section we find it more useful to use graded modules $A/xA$ instead of dg modules $M_x$, but the reader should keep in mind that they define isomorphic objects of the category $\Perf A$.

In geometrical terms, dg modules $M_x\in\Perf A$ correspond under the equivalence
$K'\colon D^b(R\mmod)\to (\Perf A)^{op}$ to complexes $F\in D^b(R\mmod)$ such that $$\sum_i \dim H^i(F)=2.$$

Some part of the work is done in a more general setting of graded rings with unique prime decomposition, so called \emph{rigid UFD's} (see Section~\ref{section_UFD}).  Although, the main results of this section hold only for free graded algebras. For  convenience of the reader we try to use different notation: a general graded algebra is denoted by $\Lambda$, whereas  $A$ stands for a free graded algebra.

\begin{lemma} 
\label{lemma_shortexact} 
Let $\Lambda $ be a graded ring having no zero divisors, let $a,b\in \Lambda $ be homogeneous elements, $\deg (a)=d$. Then we have the natural short exact sequence of graded $\Lambda $-modules
$$0\to \Lambda /b\Lambda [-d]\stackrel{f}{\to}\Lambda /ab\Lambda \stackrel{g}{\to }\Lambda /a\Lambda \to 0$$
where $f$ is the left multiplication by $a$ and $g$ is the projection.
\end{lemma}
\begin{proof}
Let 
$g\colon \Lambda /ab\Lambda\to \Lambda /a\Lambda$
be the projection, then $\ker g\cong a\Lambda /(ab\Lambda )$. Left multiplication by $a$
gives an isomorphism  $\Lambda /b\Lambda [-d]\to a\Lambda /(ab\Lambda )$ of graded $\Lambda$-modules (since $\Lambda$ has no zero divisors). 
\end{proof}

\begin{lemma}
\label{lemma_Mg}
Let $A$ be a free graded algebra, let $a_1,a_2,a_3\in A^*$ be homogeneous elements. Then
\begin{enumerate}
\item $A/{a_1a_2}A\in \langle A/{a_1}A,A/{a_2}A\rangle$;
\item $A/{a_1}A\in \langle A/{a_1a_2}A,A/{a_2}A\rangle$,
 $A/{a_2}A\in \langle A/{a_1a_2}A,A/{a_1}A\rangle$ ;
\item $\langle A/{a_1a_2}A,A/{a_2a_3}A\rangle=\langle A/{a_1}A,A/{a_2}A,A/{a_3}A\rangle$;
\item $\langle A/{a_1^n}A\rangle=\langle A/{a_1}A\rangle$ for any $n\ge 1$.
\end{enumerate}
\end{lemma}

\begin{proof}
Denote $d_i=\deg a_i$.

Lemma \ref{lemma_shortexact} immediately implies (1),(2) and $\subset$ parts of (3) and (4).

For (3), consider the homomorphism
$$f\colon A/(a_2a_3A)[-d_1]\xra{a_1\cdot} A/(a_1a_2A)$$
given by left multiplication by $a_1$.
One has
$$\ker f=a_2A/(a_2a_3A)[-d_1]\cong A/a_3A[-d_1-d_2]\quad \text{and} \quad \coker f=A/a_1A$$
By Proposition~\ref{prop_module}, the cone of $f$  is isomorphic in $\Perf A$ to its cohomology, which is
$(\ker f)[1]\oplus \coker f$. Thus $A/{a_1}A$ and $A/{a_3}A$ belong to the l.h.s. in (3). By (2) $A/{a_2}A$ also does. Hence $\supset$ in (3) holds.

(4) follows from (3) by taking $a_3=a_1, a_2=a_1^{n-1}$.
\end{proof}

\begin{remark}
Note that  Lemma~\ref{lemma_Mg} also follows readily from Corollary~\ref{cor_G1G2} and Lemma~\ref{lemma_g1g2g3}  from Section~\ref{section_localization}.
\end{remark}

\subsection{Free algebras as rigid UFD's}
\label{section_UFD}

Here we recall some notions related to prime decomposition in noncommutative rings. We refer to \cite{Co2} for the terminology.

Let $\Lambda$ be a graded integral domain: that is, a graded associative unital ring with no zero divisors. The set $\Lambda^*$ of nonzero homogeneous elements in $\Lambda$ is a multiplicative semigroup. Denote by $U(\Lambda)\subset \Lambda$ the  set of units. We have $U(\Lambda)\subset \Lambda^*$ since $\Lambda$ is a domain.

Elements $a,b\in\Lambda^*$  are called \emph{right} (resp. \emph{left}) \emph{associated} if $a=bu$ (resp. $a=ub$) for some unit $u\in\Lambda$. Note that right and left associatedness are the same if and only if the subgroup $U(\Lambda)\subset \Lambda^*$ is normal. If this is the case, we will simply use the term \emph{associated}.

An element $a\in\Lambda^*$ is called \emph{prime} if  $a$ is not a unit and $a$ is not a product of two homogeneous non-units.

\begin{definition}
\label{def_strong}
Let $\Lambda$ be a graded integral domain. We say that $\Lambda$ is a \emph{graded rigid unique factorization domain} (or a \emph{graded rigid UFD}) if
\begin{itemize}
\item the subgroup $U(\Lambda)\subset \Lambda^*$ is normal
\item the semigroup $\Lambda^*/U(\Lambda)$ is (noncommutative) free.
\end{itemize}
\end{definition}

\begin{remark}
By definition, in rigid UFD's elements are right associated if and only if they are left associated.
\end{remark}

Clearly, nice commutative rings like integers or polynomials over a field are not rigid UFD's. What is important for us, free graded algebras are graded rigid UFD's  by \cite[Prop. 6.6.3]{Co}.

One has the following, see \cite[Th. 7.1]{Co2} for the non-graded case.
\begin{prop}
\label{prop_UFD}
Let $\Lambda$ be a graded rigid UFD. Then any non-unit in $\Lambda^*$ has a factorization into homogeneous primes. If an element $a\in\Lambda^*$ has two such prime factorizations
$$a=p_1\ldots p_r=q_1\ldots q_s$$
then $r=s$ and $p_i$ is associated with $q_i$ for any $i$.
\end{prop}

\begin{definition}
Let $\Lambda$ be a graded rigid UFD and $a\in \Lambda^*$.
We denote the number  of prime factors of $a$ by $l(a)$ and call it the \emph{length} of $a$. Clearly,
length  has the following properties:
\begin{itemize}
\item $l(a)\ge 0$; $l(a)=0$ if and only if $a$ is a unit; $l(a)=1$ if and only if $a$ is prime;
\item $l(ab)=l(a)+l(b)$.
\end{itemize}
\end{definition}

The following lemma is an easy consequence  of  Proposition~\ref{prop_UFD}.

\begin{lemma}
\label{lemma_abcd}
Let $\Lambda$ be a graded rigid UFD (for example, $\Lambda$ can be a free graded algebra). Suppose $a,b,c,d\in\Lambda^*$ and $ab=cd$. Then either
\begin{enumerate}
\item $a=ce, d=eb$ for some $e\in \Lambda^*$,  or
\item $c=ae, b=ed$ for some $e\in \Lambda^*$.
\end{enumerate}
Moreover, if $l(a)\ge l(c)$ then (1) holds, if $l(a)\le l(c)$ then (2) holds and  if $l(a)=l(c)$ then $e$ is a unit.
\end{lemma}

The following proposition is given as a motivation for Definition \ref{def_goodelement} below. We will use only part (2) of this proposition.

\begin{prop} 
\label{prop_modulesfromfactorizations} 
Let $\Lambda$ be a graded rigid UFD (for example, $\Lambda$ can be a free graded algebra). Fix $x,y\in \Lambda ^*$ and consider the graded modules $\Lambda /x\Lambda $ and $\Lambda /y\Lambda$. Then the following holds:
\begin{enumerate}
\item $\Hom^{\bul}_{\grmod \Lambda}(\Lambda /x\Lambda ,\Lambda /y\Lambda)\neq 0$ if and only if $x=ca,\ y=bc$ for some $a,b,c\in \Lambda ^*$ with $l(c)>0$.

\item $\Lambda /x\Lambda $ is a submodule of $\Lambda /y\Lambda [i]$ for some $i$ if and only if $y=dx$ for some $d\in \Lambda ^*$. If it is the case, the quotient module is isomorphic to $\Lambda/d\Lambda[i]$.

\item  $\Lambda /x\Lambda $ has a quotient module $\Lambda /y\Lambda [i]$    for some $i$ if and only if $x=yd$ for some $d\in \Lambda ^*$. If it is the case, the kernel is isomorphic to $\Lambda/d\Lambda$.

\item $\Lambda /x\Lambda \cong \Lambda /y\Lambda [i]$ for some $i$ if and only if $x$ and $y$ are associated. 
\end{enumerate}
\end{prop}

\begin{proof} 
(1) First, suppose $x=ca,\ y=bc$ for some $a,b,c\in \Lambda ^*$. Consider the following composition of natural maps:
\begin{equation}
\label{eq_fprime}
\xymatrix{
\Lambda/x\Lambda = 
\Lambda/ca\Lambda \ar@{->>}[r] &
\Lambda/c\Lambda \ar[r]^-b_-{\sim}  & 
b\Lambda/bc\Lambda[i] =
b\Lambda/y\Lambda[i] \ar@{_(->}[r] &
\Lambda/y\Lambda[i]
}
\end{equation}
where $i=\deg b$. The composition is nonzero provided that $l(c)>0$, this proves ``if'' part. Now assume $f\colon \Lambda/x\Lambda\to \Lambda/y\Lambda[i]$ is a nonzero homomorphism. Choose $b\in \Lambda^*$ such that $f(1+x\Lambda)=b+y\Lambda$. Then necessarily $bx\in y\Lambda$, $bx=ya$ for some $a\in \Lambda^*$.
By Lemma \ref{lemma_abcd} we have either
\begin{itemize}
\item $b=yc,\ a=cx$ for some $c\in \Lambda ^*$ or

\item $y=bc,\ x=ca$ for some $a\in \Lambda ^*$.
\end{itemize}
The first case implies that the map $f$ is zero, which contradicts our assumptions. So we are in the second case. Let $f'$ be the homomorphism in \eqref{eq_fprime}. Then $f=f'$ since both maps send $1+x\Lambda$ to $b+y\Lambda$.
Since $f\ne 0$, we have $l(c)>0$. This finishes the proof of  (1). 

Part ``if'' of (2) follows from \eqref{eq_fprime} by taking $c=x$, $a=1$, $b=d$.
For ``only if'', use decomposition \eqref{eq_fprime} of an injective homomorphism and note that $ca\Lambda=c\Lambda$. Hence $a$ is a unit and $x$ is right associated to $c$. By the definition of UFD, $x$ is left associated to $c$: $c=a'x$ for some unit $a'$. Then $y=bc=(ba')x$, take $d:=ba'$. Clearly, the quotient is $\Lambda/b\Lambda[i]=\Lambda/d\Lambda[i]$.

(3) is proved similarly to (2), (4) follows from (2) and (3).
\end{proof}

\subsection{Good sets and good elements}



We have found the following notion very useful. We work here in the setup of graded rigid UFD's (see Section~\ref{section_UFD}) but the only examples we keep in mind are free graded algebras.

\begin{definition}
\label{def_goodelement}
Let $\Lambda$ be a graded rigid UFD, let $\XX\subset \Lambda^*$ be a subset. We say that $\XX$ is \emph{good} if there are no units in $\XX$ and for any $x,y\in\XX$ and any elements $a,b,c\in \Lambda^*$ such that
$$x=ab,\quad y=ca$$
either
\begin{itemize}
\item $a$ is a unit, or
\item $x=y$ and $b$ and $c$ are units.
\end{itemize}
We say that an element $x\in \Lambda^*$ is \emph{good}
if the set $\{x\}$  is good.
\end{definition}

\begin{example}
\label{example_goodbad}
Let $A=\k\{x,y,\ldots\}$ be a free algebra.
\begin{enumerate}
\item Elements  $x,xy,x^3y^4, x^2y^3xy, x^2yxy^2$ are good while elements
$x^2, x^3, xyx, xyxy, xy^2x^2y, $ are not.
\item  The sets $\{x,y\},\{xy,x^2y^2\}, \{x^2yxy^2, x^3y^2x^2y^3, x^4y^3x^3y^4\}$ are good while the sets $\{xy, yx^2\},\{x^2y^3,x^3y^2\}, \{x^2yxy^2, xy^4xy\}, \{xyx^2y^2, x^2y^2xy\}$ are not.
\end{enumerate}
\end{example}

For the future use we prove some easy lemmas.
\begin{lemma}
\label{lemma_yzy}
Let $\Lambda$ be a graded rigid UFD, let $x\in \Lambda^*$ be an element which is not good. Then there exist elements $y,z\in \Lambda^*$ with $l(y)>0$ such that $x=yzy$.
\end{lemma}
\begin{proof}
By the definition, there exist elements $a,b,c\in \Lambda^*$ with $l(a),l(b),l(c)>0$ such that $x=ab=ca$. Let $n\ge 0$ be the maximal such that $a=db^n$ for some $d\in \Lambda^*$. We have $x=db^{n+1}=cdb^n$ and $db=cd$. If $l(d)\ge l(b)$ then  by Lemma~\ref{lemma_abcd} we have $d=eb$ for some $e\in \Lambda^*$, $a=eb^{n+1}$ and $n$ is not maximal, a contradiction. Hence $l(d)<l(b)$. It follows from Lemma~\ref{lemma_abcd} that $b=ed, c=de$ for some $e\in \Lambda^*$. Therefore
$$x=cdb^n=ded(ed)^n=d(ed)^{n+1}.$$
If $l(d)>0$ then we can take $y:=d$, $z=(ed)^ne$. If
$l(d)=0$ then $d$ is a unit, $l(e)=l(b)-l(d)>0$, $n\ge 1$ and we can take
$y:=ded, z:=(ed)^{n-1}d^{-1}$.
\end{proof}

\begin{lemma}
\label{lemma_monomialfaith}
Let $\XX$ be a good set in a graded rigid UFD $\Lambda$.
\begin{enumerate}
\item If two elements $a,b\in\XX$ are associated then $a=b$.
\item Let
$$m_1=x_1\ldots x_r,\quad m_2=z_1\ldots z_s, \quad x_i,z_i\in\XX$$
be two monomials in $\XX$.  Assume $m_1$ is associated with $m_2$, then $m_1=m_2$, $r=s$ and $x_i=z_i$ for all $i$.
\end{enumerate}
\end{lemma}
\begin{proof}
(1) We have $a=bu$ for some unit $u\in\Lambda^*$. Since $a=bu, b=1\cdot b\in\XX$, by Definition~\ref{def_goodelement} we get either that $b$ is a unit (contradiction to Definition~\ref{def_goodelement}) or that $a=b$.

(2) We have
\begin{equation}
\label{eq_xzu}
x_1\ldots x_r=z_1\ldots z_su
\end{equation}
for some unit $u\in\Lambda^*$.
It follows from Lemma~\ref{lemma_abcd} that either $x_1=z_1y$ or $z_1=x_1y$ for some $y\in\Lambda^*$. Since $x_1,z_1$ are not units,  Definition~\ref{def_goodelement} implies that  $x_1=z_1$.
Dividing \eqref{eq_xzu} by $x_1$ from the left, we continue the procedure and get that $r=s$, $x_i=z_i$ for all $i$ and hence $u=1$. Thus $m_1=m_2$.
\end{proof}

\begin{lemma}
\label{lemma_goodx}
Let $\Lambda$ be a graded rigid UFD and $\XX\subset \Lambda^*$ be a good subset. Let $x,y\in \XX$.
Assume  $ax\in y\Lambda$ for some $a\in \Lambda^*$. Then either
\begin{enumerate}
\item $a$ is a unit and $x=y$, or
\item $a\in y\Lambda$.
\end{enumerate}
\end{lemma}
\begin{proof}
We have $ax=yb$ for some $b\in \Lambda^*$.

If $l(a)< l(y)$ then by Lemma~\ref{lemma_abcd} we have $y=az$, $x=zb$  for some $z\in \Lambda^*$. Since $\XX$ is good, we get either that $z$ is a unit (this is impossible since $l(z)=l(y)-l(a)>0$) or $x=y$ and both $a$ and $b$ are units (this is alternative (1)).

If $l(a)\ge l(y)$ then by Lemma~\ref{lemma_abcd} we get $a=yc$ for some $c\in \Lambda^*$, hence $a\in y\Lambda$ and alternative (2) holds.
\end{proof}

\medskip
Let $\Lambda$ be a graded $\k$-algebra.
Recall that for graded $\Lambda$-modules $M$ and $N$ we denote
$$\Hom_{\grmod \Lambda}^{\bul}(M,N):=\oplus_i \Hom_{\grmod \Lambda}(M,N[i]).$$

We finish this section with an easy but important calculation for free algebras.

\begin{lemma}
\label{lemma_AxAAxA}
Let $A$ be a free graded $\k$-algebra.
Let $\XX\subset A^*$ be a good subset and $x,y\in \XX$. Then
\begin{align*}
\Hom_{\grmod A}(A/xA,A/yA[i])&=\begin{cases} \k, & x=y, i=0;\\
0, & \text{otherwise}. \end{cases}, \\
\Hom_{\grmod A}^{\bul}(A/xA,A/yA)&=\begin{cases} \k, & x=y;\\
0, & \text{otherwise}. \end{cases}
\end{align*}
\end{lemma}
\begin{proof}
Recall that $A$ is a graded rigid UFD. Also, an element $a\in A^*$ is a unit iff $\deg a=0$ and iff $a\in\k^*$.

Let $f\in \Hom_{\grmod A}(A/xA,A/yA[i])$. Then as in the proof of Proposition \ref{prop_modulesfromfactorizations} $f$ is the left multiplication by a homogeneous element $a\in A_i$ of degree $i$ such that $ax\in yA$. By Lemma~\ref{lemma_goodx}, either $x=y$ and $a$ is a unit (then $i=0$, $a\in\k^*$ and $f$ is a scalar endomorphism) or $a\in yA$ and $f=0$.
\end{proof}

\subsection{$\XX$-filtrations}

For a set $\XX$ of elements in an algebra $\Lambda$, we  introduce the notion of $\XX$-filtration on a $\Lambda$-module.
For a free graded algebra $A$ we relate $\XX$-filtrations of a module with its generation by modules of the form $A/xA$, $x\in\XX$.

\begin{definition}
\label{def_filt}
Let $\Lambda$ be a graded algebra and  $\XX\subset \Lambda^*$ be a subset. Let us say that a graded $\Lambda$-module $M$ is \emph{$\XX$-filtered} if there exists a finite filtration
$$0=F_0M\subset F_1M\subset F_2M\subset \ldots \subset F_nM=M$$
by graded submodules such that any quotient $F_mM/F_{m-1}M$ is isomorphic to some shift of the module $\Lambda/x\Lambda$ for some $x\in \XX$.
\end{definition}

\begin{prop}
\label{prop_xprod}
Let $\Lambda$ be a graded rigid UFD (for example, a free graded algebra).
Let $\XX\subset \Lambda^*$ be a family and $x\in \Lambda^*$ be an element. Then
$\Lambda/x\Lambda$ is $\XX$-filtered if and only if
\begin{equation}
\label{eq_xprod}
x=\lambda\cdot \prod_{i=1}^n y_i\quad\text{for some}\quad \lambda\in U(\Lambda), \quad y_i\in\XX \quad \text{and}\quad n\ge 0.
\end{equation}
\end{prop}
\begin{proof}
The ``if'' part follows from  Lemma \ref{lemma_shortexact}. Let us prove the ``only if'' part.

Assume that the module $\Lambda/x\Lambda$ has an  $\XX$-filtration of length $n$.
If $n=0$ then $x$ is a unit. If $n>0$, there exists $y\in \XX$ and $i$ such that $\Lambda /y\Lambda [i]$ is a submodule of $\Lambda /x\Lambda$. Then by part (2) of Proposition \ref{prop_modulesfromfactorizations} $x=dy$ for some $d\in \Lambda ^*$ and
$$(\Lambda /x\Lambda)/(\Lambda /y\Lambda[i])\cong \Lambda /d\Lambda.$$
So replacing $x$ by $d$ we can proceed by induction on $n$.
\end{proof}

Now  we concentrate on free graded algebras (generated by elements of degree $1$). Recall that a free graded algebra  is a graded rigid UFD.

\begin{lemma}
\label{lemma_xfilt}
Let $A$ be a free graded algebra.
Assume that $\XX\subset A^*$ is a good subset and the  graded $A$-module $M$ has an $\XX$-filtration of length $n$. Let $x\in \XX$ be an element. Then there exist homogeneous elements $e_1,\ldots, e_m\in M$ for some $0\le m\le n$ and an $\XX$-filtration $F_iM$ of $M$ of length $n$ such that the following holds.
\begin{enumerate}
\item For any $i$ the annihilator of $e_i$ in $A$ is $xA$. Denote $P_i=e_iA\subset M$, this submodule is isomorphic to $A/xA[d_i]$ where $d_i=-\deg e_i$.
\item The sum $\oplus_{i=1}^m P_i$ is direct.
\item For any $i=1,\ldots,m$ one has $F_iM=\oplus_{j=1}^iP_j$.
\item $\dim_{\k}\Hom_{\grmod A}^{\bul}(A/xA,M)=m$.
\end{enumerate}
\end{lemma}
\begin{proof}
We argue by induction in $n$. The case $n=0$ is trivial. For $n\ge 1$, assume $M$ has an $\XX$-filtration of length $n$. Then there exists an exact sequence
\begin{equation}
\label{eq_MM}
0\to M'\to M\to A/yA[d]\to 0
\end{equation}
for some $y\in\XX$, $d\in\Z$ and some graded $A$-module $M'$ admitting an $\XX$-filtration of length $n-1$.
By the induction assumption for some $0\le m\le n-1$ we can choose elements $e'_1,\ldots,e'_m\in M'$ and an $\XX$-filtration
$F_iM'$ of $M'$ satisfying properties (1)-(4).
Consider the exact sequence associated with \eqref{eq_MM}
$$0\to \Hom_{\grmod A}^{\bul}(A/xA,M')\to \Hom_{\grmod A}^{\bul}(A/xA,M)\xra{\beta}
\Hom_{\grmod A}^{\bul}(A/xA,A/yA[d]).$$

If $\beta=0$ then $\dim_{\k}\Hom_{\grmod A}^{\bul}(A/xA,M)=\dim_{\k}\Hom_{\grmod A}^{\bul}(A/xA,M')=m$ and we can take $e'_i=e_i$  and $F_iM=F_iM'$ for $1\le i\le n-1$, $F_nM=M$. Clearly properties (1)-(4) hold.

Suppose $\beta\ne 0$. It follows then from  Lemma~\ref{lemma_AxAAxA} that  $\Hom_{\grmod A}^{\bul}(A/xA,A/yA[d])=\k$, $x=y$ and $\beta$ is surjective.
Hence the sequence \eqref{eq_MM} splits and $M\cong M'\oplus A/xA[d]$. Then $\dim_{\k}\Hom_{\grmod A}^{\bul}(A/xA,M)=\dim_{\k}\Hom_{\grmod A}^{\bul}(A/xA,M')+1=m+1$. We can take $e_1$ to be the generator of
$A/xA[d]$, $e_{i}=e'_{i-1}$ for $i=2,\ldots,m+1$. Also we take  $F_{i+1}M=A/xA[d]\oplus F_iM'$ for $i=0,\ldots,n-1$. Again, properties (1)-(4) hold.
\end{proof}

\begin{lemma}
\label{lemma_xfilt2} Let $A$ be a free graded algebra.
Assume $\XX\subset A^*$ is a good subset, $x\in \XX$ and a  graded $A$-module $M$ has an $\XX$-filtration of length $n$. Let $f\colon A/xA[d]\to M$ be a homomorphism of graded $A$-modules, $f\ne 0$. Then $f$ is injective and there exists an $\XX$-filtration $F'_iM$ of $M$ of length $n$ such that
$F'_1M=\im f$.
\end{lemma}
\begin{proof}
Choose an $\XX$-filtration $F_iM$ and elements $e_1,\ldots,e_m$ as in Lemma~\ref{lemma_xfilt}. Obviously,
$$\Hom_{\grmod A}^{\bul}(A/xA[d],M)=\oplus_i \{w\in M_i\mid wx=0\}.$$
We claim that
\begin{equation}
\label{eq_eee}
\oplus_i \{w\in M_i\mid wx=0\}=\langle e_1,\ldots, e_m\rangle_{\k},
\end{equation}
where $\langle \ldots \rangle_{\k}$ denotes the $\k$-linear span. Indeed, any $e_i$ belongs to the l.h.s. of \eqref{eq_eee} by Lemma~\ref{lemma_xfilt}(1).  On the other hand, the vectors $e_1,\ldots,e_m$ are linearly independent by Lemma~\ref{lemma_xfilt}(2), hence the dimension of both sides of \eqref{eq_eee} is $m$ (see Lemma~\ref{lemma_xfilt}(4)) and the equality holds.

Now let $e\ne 0$ be the image of $1\in A/xA[d]$ under $f$. Clearly, $ex=0$, hence $e\in e_1,\ldots, e_m\rangle_{\k}$. Write $e=\sum_{l=1}^{m_0} c_le_{i_l}$ where $e_{i_1},\ldots,e_{i_{m_0}}$ are homogeneous of degree $-d$, $0\ne c_l\in \k$.
Renumbering $e_1,\ldots,e_m$, we get a new sequence $e'_1,\ldots,e'_m$ such that $e'_1=e_{i_1},\ldots,e'_{m_0}=e_{i_{m_0}}$. Now we can replace $e_1$ by $e$. Put $P_1=eA, P_i=e'_iA$ for $i=2,\ldots,m$. From Lemma~\ref{lemma_xfilt}(2) it follows that the sum $\oplus_{i=i}^m P_i$ is direct. Put
$$F'_iM=\oplus_{j=1}^iP_j\quad\text{for}\quad 1\le i\le m, \qquad
F'_iM=F_iM\quad\text{for}\quad m\le i\le n,$$
one checks that $F'_iM$ is also an $\XX$-filtration. In particular, $f$ gives an isomorphism $A/xA[d]\to eA=P_1=F'_1M\subset M$.
\end{proof}

\subsection{$\XX$-filtrations and generation}

\begin{lemma}
\label{lemma_Gdiag}
Suppose $\XX\subset A^*$ is a good family of elements of a free graded algebra $A$. Then a graded $A$-module $M$ belongs to the subcategory $\langle A/xA\rangle_{x\in\XX}\subset \Perf A$ if and only if $M$ is $\XX$-filtered.
\end{lemma}
\begin{proof}
Part ``if'' is more or less obvious. We prove the ``only if'' part.

First, we assume that $M\in [A/xA]_{x\in\XX}$. Then we can argue by induction in the number $n$ of copies of shifted modules $A/xA, x\in\XX$ that the module $M$ is built from. For the induction step, assume $M$ is quasi-isomorphic the cone $C$ of a morphism $f\colon A/xA[d]\to  M'$ in $\Perf A$, where  $x\in\XX$, $d\in\Z$ and the object $M'\in\Perf A$ is built from $n-1$ copies   of shifted modules of the form $A/yA, y\in\XX$. By Proposition~\ref{prop_module}, we can assume that $M'$ is a graded module and by the induction assumption $M'$ is $\XX$-filtered. Consider the long exact sequence in cohomology
$$A/xA[d] \xra{H(f)} M'\xra{\delta} H(C)\to A/xA[d+1]\xra{H(f)} M'[1].$$
Clearly, the graded modules $M=H(M)$ and $H(C)$ are isomorphic.

Suppose $H(f)=0$, then we get an exact sequence
$$0\to M'\to H(C)\to A/xA[d+1]\to 0,$$
it follows that the graded $A$-module $H(C)$ is $\XX$-filtered.

Suppose $H(f)\ne 0$. Then by Lemma~\ref{lemma_xfilt2} $H(f)$ is injective and there is an $\XX$-filtration $F_iM'$ of $M'$ such that $\im H(f)=F_1M'$.
There is an exact sequence
$$0\to A/xA[d]\xra{H(f)} M' \xra{\delta} H(C)\to 0,$$
hence $\delta$ induces an $\XX$-filtration on $H(C)$.

Now we consider the general case when $M\in \langle A/xA\rangle_{x\in\XX}$. It means that $M$ is a direct summand of some object in $[A/xA]_{x\in\XX}$. By the above special case and Proposition~\ref{prop_module}, we can assume that $M$ is a direct summand of some graded $\XX$-filtered  $A$-module $N$.

We prove by induction in $n$ that for a graded  $A$-module $N=M\oplus M'$ with an $\XX$-filtration of length $n$ the modules $M$ and $M'$ are also $\XX$-filtered.
The base $n=0$ is trivial.

Now assume that $N=M\oplus M'$ has an $\XX$-filtration $F_iN$ of length $n\ge 1$. We have $F_1N\cong A/xA[d]$ for some $x\in\XX$, $d\in \Z$. Hence $\Hom_{\grmod A}(A/xA[d],N)\ne 0$. Further, we have
\begin{equation}
\label{eq_NMM}
\Hom_{\grmod A}(A/xA[d],N)\cong \Hom_{\grmod A}(A/xA[d],M)\oplus \Hom_{\grmod A}(A/xA[d],M').
\end{equation}
It follows that at least one summand in the r.h.s. of \eqref{eq_NMM} is nonzero, we can assume that $\Hom_{\grmod A}(A/xA[d],M)\ne 0$. Choose a nonzero homomorphism $f\colon A/xA[d]\to M$ of graded modules.
Then $f\colon A/xA[d]\to N$ is also nonzero. By Lemma~\ref{lemma_xfilt2}, there exists an $\XX$-filtration $F'_iN$ of length $n$ such that $F'_1N=\im f\cong A/xA[d]$. It follows that $F'_1N\subset M$. Passing to the quotients, we have $N/F'_1N=(M/F'_1N)\oplus M'$ and $N/F'_1N$ has an $\XX$-filtration of length $n-1$. By the induction assumption, we get that  $M/F'_1N$ and $M'$ are $\XX$-filtered. Since $F'_1N\cong A/xA[d]$, the  module $M$ is also $\XX$-filtered.
\end{proof}

\begin{lemma}
\label{lemma_Ggivesx}
Let $x\in A$ be a good element of a free graded algebra and $M$ be a nonzero graded $\{x\}$-filtered $A$-module. Then $A/xA\in\langle M\rangle$.
\end{lemma}
\begin{proof}
We argue by induction in the length $n$ of  $\{x\}$-filtration.
For $n=1$ we have $M\cong A/xA[d]$ for some $d$ and there is nothing to prove.
In general,  let $0\subset F_1M\subset \ldots\subset F_{n-1}M\subset F_nM=M$ be the filtration such that $F_iM/F_{i-1}M\cong A/xA[d_i]$ for $i=1,\ldots,n$ and some integers $d_i$. Consider the composition $f$ of the following maps
$$M\to M/F_{n-1}M\cong A/xA[d_n]=A/xA[d_1][d_n-d_1]\cong F_1M[d_n-d_1]\to M[d_n-d_1].$$
By Proposition~\ref{prop_module}, the cone $C$ of $f$ is quasi-isomorphic to the direct sum $\coker f\oplus (\ker f)[1]$. It follows that
$$F_{n-1}M \cong \ker f \in \langle M\rangle.$$
Since the module $F_{n-1}M$ has an $\{x\}$-filtration of length $n-1$, by the induction assumption $A/xA\in \langle F_{n-1}M\rangle$ and we are done.
\end{proof}

Lemmas \ref{lemma_Gdiag} and \ref{lemma_Ggivesx} combine into
\begin{prop}
\label{prop_Mxminimal}
Let $x\in A$ be a good element in a free graded algebra. Then the category $\langle A/xA\rangle$ has no non-trivial thick triangulated subcategories.
\end{prop}

Using Proposition~\ref{prop_xprod} we get
\begin{theorem}
\label{theorem_xprod}
Let $A$ be a free graded algbera.
Let $\XX\subset A^*$ be a good subset and $x\in A^*$ be an element. Then
$$A/xA\in\langle A/yA\rangle_{y\in\XX}$$
if and only if
\begin{equation*}
x=\lambda\cdot \prod_{i=1}^n y_i\quad\text{for some}\quad \lambda\in\k^*, \quad y_i\in\XX \quad \text{and}\quad n\ge 0.
\end{equation*}
\end{theorem}
\begin{proof}
Use Proposition~\ref{prop_xprod} and Lemma~\ref{lemma_Gdiag}.
\end{proof}

One can generalize Theorem~\ref{theorem_xprod} to arbitrary (not necessarily good) subsets.
To do this, we need to substitute an arbitrary subset in $A^*$ by a good one.

\begin{prop}
\label{prop_makegood}
Let $A$ be a free graded algebra and $\XX\subset A^*$ be a subset. Then there exists a  good subset $\~\XX\subset A^*$ such that
\begin{equation}
\langle A/xA\rangle_{x\in\XX}=\langle A/xA\rangle_{x\in\~\XX}\quad \text{as subcategories of $\Perf A$}.
\end{equation}
\end{prop}
%
\begin{proof} 

We put $T:=\langle A/xA\rangle_{x\in\XX}$. This is a thick triangulated subcategory of $\Perf A$. Define
$$\YY :=\{x\in A^*\mid A/xA\in T\}$$
and
$$\ZZ :=\{x\in \YY \mid \text{if $x=yz$ with $y,z\in \YY$, then $y$ or $z$ is a unit in $A$}\}$$
Obviously $\langle A/xA\rangle_{x\in\YY}=T$ and $\langle A/xA\rangle_{x\in\ZZ}\subset T$. We claim that in fact $\langle A/xA\rangle_{x\in\ZZ}=T$. Indeed, it suffices to prove that for every $y\in \YY$ we have $A/yA\in \langle A/xA\rangle_{x\in\ZZ}$. If $y\in \ZZ$ there is nothing to prove. Otherwise, there exists a factorization $y=y_1y_2$, with $y_1,y_2\in \YY$ and $l(y_1),l(y_2)>0$. It follows from Lemma \ref{lemma_Mg}(1) that $A/yA\in
\langle A/xA\rangle_{x\in\ZZ}$ if $A/y_1A,A/y_2A\in \langle A/xA\rangle_{x\in\ZZ}$. Hence we are done by induction on  $l(y)$.

For $x,y\in A^*$, put $x\sim y$ if $x$ and $y$ are associated. Choose a subset $\~\XX\subset \ZZ$ which contains exactly one element from every $\sim$ equivalence class in $\ZZ$. Clearly
$$\langle A/xA\rangle_{x\in\~\XX}=\langle A/xA\rangle_{x\in\ZZ}=T$$

We claim that the set $\~\XX$ is good. Indeed, assume by contradiction that there exist elements $x,y\in \~\XX$ such that the set $\{x,y\}$ is not good.

Case (1): $x=y$. Then by Lemma \ref{lemma_yzy} there exist $s,t\in A^*$ with $l(s)>0$ and $x=sts$. It follows from Lemma~\ref{lemma_Mg}(3) (take $a_1=st,a_2=s,a_3=ts$) that 
$A/sA,A/tsA,A/stA\in \langle A/stsA\rangle \subset T$. It follows that $s,ts\in \YY$ and hence $x=sts\notin \ZZ$ --- a contradiction.

Case (2): $x\ne y$. In this case we have
$$x=ab,\quad y=ca$$
such that $l(a)>0$. 
As above, by Lemma~\ref{lemma_Mg}(3) 
$A/aA,A/bA,A/cA\in  T$. It follows that $a,b,c\in \YY$. By definition of $\ZZ$ we have that $b,c$ are units. Hence $x$ and $y$ are associated, a contradiction.

So the set $\~\XX$ is good, which completes the proof of
Proposition \ref{prop_makegood}.
\end{proof}

\begin{example}
For the sets from Example~\ref{example_goodbad} we have
\begin{enumerate}
\item
$\~{\{x^2\}}=\~{\{x^3\}}=\{x\}$;
\item $\~{\{xyx\}}=\~{\{xy^2x^2y\}}=\~{\{xy, yx^2\}}=\~{\{x^2y^3,x^3y^2\}}=\~{\{x^2yxy^2, xy^4xy\}}=\{x,y\}$;
\item $\~{\{xyxy\}}=\{xy\}$;
\item $\~{\{xyx^2y^2, x^2y^2xy\}}=\{xy,x^2y^2\}$.
\end{enumerate}
\end{example}

Combining Theorem~\ref{theorem_xprod} and Proposition~\ref{prop_makegood}, we get
\begin{theorem}
\label{theorem_xprod2}
Let $A$ be a free graded algebra, let $\XX\subset A^*$ be a subset and $x\in A^*$ be an element. Let $\~ \XX\subset A^*$ be the good set constructed in Proposition~\ref{prop_makegood}. Then
$$A/xA\in\langle A/yA\rangle_{y\in\XX}$$
if and only if
\begin{equation*}
x=\lambda\cdot \prod_{i=1}^n y_i\quad\text{for some}\quad \lambda\in\k^*, \quad y_i\in\~\XX \quad \text{and}\quad n\ge 0.
\end{equation*}
\end{theorem}

\begin{corollary}
\label{cor_Mgproper}
Let $A$ be a free graded algebra.
For any family $\XX\subset A^*$ the category
$\langle A/xA\rangle_{x\in \XX}$ is not equal to all $\Perf A$.
\end{corollary}
\begin{proof}
Let $\~\XX$ be the good set from Proposition~\ref{prop_makegood}.
It suffices to check that $A\notin \langle A/xA\rangle_{x\in \XX}=\langle A/xA\rangle_{x\in \~\XX}$. By Lemma~\ref{lemma_Gdiag}, it suffices to check that $A$ admits no $\~\XX$-filtration. Indeed, all homomorphisms $A/xA[d]\to A$ of graded $A$-modules are zero.
\end{proof}

\section{Lattice of subcategories}
\label{section_lattice}

As usual, we denote by $R_N$ the algebra $\k[y_1,\ldots,y_N]/(y_1,\ldots,y_N)^2$.

In this section we discuss some properties of lattice of thick triangulated subcategories in $D^b(R_N\mmod)$ and compare it to ones for artinian complete intersections studied in \cite{CI}. The following statement is an easy corollary of  Theorem~\ref{theorem_CI}.

\begin{prop}
\label{prop_lattice1}
Let $R$ be a complete intersection ring of dimension zero (see Theorem~\ref{theorem_CI}). Then the lattice~$\LL(R)$ of nonzero finitely generated thick triangulated subcategories in $D^b(R\mmod)$ has the following properties:
\begin{enumerate}
\item $\LL(R)$ satisfies the descending chain condition;
\item there exists the least element $\Perf R$ in $\LL(R)$, let us call minimal elements in $\LL(R)\setminus \{\Perf R\}$ \emph{almost minimal};
\item let $T_1,T_2\in \LL(R)$ be almost minimal. Then there exists finitely many (in fact, only four) elements in $\LL(R)$ bounded above by $\langle T_1, T_2\rangle$ (that is, only four subcategories in $\langle T_1,T_2\rangle$): $\Perf R, T_1,T_2,\langle T_1,T_2\rangle$.
\end{enumerate}
\end{prop}
\begin{proof}
By Theorem~\ref{theorem_CI}, $\LL(R)$ is isomorphic to the lattice $\LL$ of nonempty closed subsets in $\Spec^*\k[\theta_1,\ldots,\theta_r]$. Hence it suffices to check the statements for $\LL$. Note that $\Spec^*\k[\theta_1,\ldots,\theta_r]$ is just $\P_\k^{r-1}$ with one ``very special'' point $x_0$ (corresponding to the augmentation ideal of $\k[\theta_1,\ldots,\theta_r]$) attached.

(1) holds because $\Spec^*\k[\theta_1,\ldots,\theta_r]$ is a Noetherian topological space. The least element in $\LL$ is $\{x_0\}$, almost minimal elements in $\LL$ are $\{x_0,P\}$
where $P\in \P^{r-1}_\k$ is a closed point. For $P_1\ne P_2\in \P^{r-1}_\k$ there are only four closed subsets in $\{x_0, P_1,P_2\}=\{x_0, P_1\}\cup \{x_0, P_2\}$, and (3) holds. For (2) it remains to note that $\Perf R$ is a minimal nonzero thick triangulated subcategory in $D^b(R\mmod)$ by Theorem~\ref{theorem_Hopkins}, hence $\Perf R$ is the least one.
\end{proof}

For the category $D^b(R_N\mmod)$ that we study, the picture is completely different.
For example, the subcategory $\Perf R_N$ is far from being the least.
\begin{theorem}
\label{theorem_perfbadR}
Let $R=R_N=\k[y_1,\ldots,y_N]/(y_1,\ldots,y_N)^2$, let $T\subset D^b(R\mmod)$ be a thick triangulated subcategory. Then the following conditions are equivalent
\begin{enumerate}
\item $T=D^b(R\mmod)$;
\item $T\ni \k$;
\item $T\varsupsetneq \Perf R$;
\item $T\cap \Perf R\ne 0$ and $T\not\subset \Perf R$;
\end{enumerate}
If $T=\langle M\rangle$ for an object $M\in D^b(R_N\mmod)$ then the above conditions are equivalent to 
\begin{enumerate}
\item[(5)] the dg algebra $\End_R(M)$ is smooth over $\k$ and $M\ne 0$.
\end{enumerate}
\end{theorem}
\begin{proof}
Recall (see Section~\ref{section_schemes}) that there is an  equivalence $K'\colon D^b(R\mmod)\to (\Perf A)^{\op}$ (where $A$ is a free graded algebra) which maps $\Perf R$ to $D_{fd}(A)\subset \Perf A$. Now the statement follows from Proposition~\ref{prop_perfbadA}.
\end{proof}

\begin{prop}
\label{prop_perfbadA}
Let $A$ be a free graded algebra, let $T\subset \Perf A$ be a thick triangulated subcategory. Then the following conditions are equivalent
\begin{enumerate}
\item $T=\Perf A$;
\item $T\ni A$;
\item $T \varsupsetneq D_{fd}(A)$;
\item $T\cap D_{fd}(A)\ne 0$ and $T\not\subset D_{fd}(A)$;
\end{enumerate}
If $T=\langle M\rangle$ for an object $M\in \Perf A$ then the above conditions are equivalent to 
\begin{enumerate}
\item[(5)] the dg algebra $\End_A(M)$ is smooth over $\k$ and $M\ne 0$.
\end{enumerate}
\end{prop}
\begin{proof}
(1) $\Longleftrightarrow$ (2) is evident.

(1) $\Longrightarrow$ (3) $\Longrightarrow$ (4) is clear.

(4) $\Longrightarrow$ (3) follows from Theorem~\ref{theorem_Hopkins}. Indeed, the intersection $T\cap D_{fd}(A)$ corresponds under equivalence $K'$ to some nonzero thick triangulated subcategory in $\Perf X$. Since $X$ is a point (as a topological space) it follows that this subcategory is all $\Perf X$. Therefore
$T\supset D_{fd}(A)$.

(3) $\Longrightarrow$ (2) follows from Lemma~\ref{lemma_dissection}.
Take an object $M\in T$, $M\notin D_{fd}(A)$.
By Proposition~\ref{prop_module} we can assume that $M$ is a graded  $A$-module satisfying assumptions of Lemma~\ref{lemma_dissection}. Hence there exists a free submodule $F\subset M$ such that $M/F\in D_{fd}(A)\subset T$. Since $M\notin D_{fd}(A)$, we get that $F$ is nonzero. As $M/F\in T$ we deduce that $F\in T$. Since the subcategory $T$ is thick and triangulated it follows that $A\in T$.

(1) $\Longrightarrow$ (5) follows from the fact that a dg enhancement of $D^b(R\mmod)$ is a smooth dg category, see \cite[Theorem 1.3]{LS} or \cite[Theorem A]{ELS}.

(5) $\Longrightarrow$ (1) follows from \cite{EL}. Indeed, if the dg algebra $\End_A(M)$ is smooth then the category  $\langle M\rangle$ has a strong generator (see for example Lemma 3.5 and 3.6 in \cite{Lu}). It follows then from
\cite[Theorem 1]{EL} that  $\langle M\rangle=0$ or $\Perf A$.
\end{proof}

\begin{prop}
\label{prop_lattice2}
Let  $R_N=\k[y_1,\ldots,y_N]/(y_1,\ldots,y_N)^2$, $N\ge 2$.
The lattice $\LL(R_N)$ of nonzero finitely generated thick triangulated subcategories in $D^b(R_N\mmod)$ has the following properties:
\begin{enumerate}
\item $\LL(R_N)$ does not satisfy the descending chain condition.
\item There exists no least element in $\LL(R_N)$. Category $\Perf R_N$ is minimal and almost maximal in $\LL(R_N)$: there are no elements between $\Perf R_N$ and $D^b(R_N\mmod)$.
\item There exist minimal elements $T_1,T_2\in \LL(R_N)$ such that there are infinitely many elements bounded above by $\langle T_1,T_2\rangle$ (that is, infinitely many  subcategories in $\langle T_1,T_2\rangle$).
\end{enumerate}
\end{prop}

To prove Proposition~\ref{prop_lattice2} we need the following Lemma.
\begin{lemma}
\label{lemma_1423}
Let $\Lambda$ be a graded rigid UFD. Suppose that elements $a\ne b\in\Lambda$ form a good set $\{a,b\}$ (see Definition~\ref{def_goodelement}) and $l(a)=l(b)$ (see Definition~\ref{def_strong}). Then the set
$$\{aba^8b^8, a^2b^2a^7b^7,a^3b^3a^6b^6, a^4b^4a^5b^5\}$$
is also good, its four elements are different and
$$l(aba^8b^8)=l(a^2b^2a^7b^7)=l(a^3b^3a^6b^6)=l(a^4b^4a^5b^5)=18\cdot l(a).$$
\end{lemma}
\begin{proof}

First, $a\ne b$ implies that the four elements $aba^8b^8, a^2b^2a^7b^7,a^3b^3a^6b^6, a^4b^4a^5b^5$ are different by Lemma~\ref{lemma_monomialfaith}. Now suppose $x,y\in \{aba^8b^8, a^2b^2a^7b^7,a^3b^3a^6b^6, a^4b^4a^5b^5\}$ and $x=cp, y=qc$. Assume that $c$ is not a unit, then $0<l(c)$. Consider two cases: $l(a)|l(c)$ or not.

If $l(a)|l(c)$ then $c$ is associated with some left submonomial $m_1$ in $x$ by Lemma~\ref{lemma_abcd}. Similarly, $c$ is associated with some right submonomial $m_2$ in  $y$.
Hence, monomials $m_1,m_2$ in variables $a,b$ are associated and by Lemma~\ref{lemma_monomialfaith}(2) $m_1=m_2$ as monomials. It follows that $a$ and $b$ occur in $m_1=m_2$ in equal quantity. It is possible only if  $m_1=a^ib^i$ for $1\le i\le 4$ or $m_1=x$. Similarly, $m_2=a^ib^i$ for $5\le i\le 8$ or $m_2=y$. Therefore
$m_1=m_2=x=y$, $p,q$ are units and we are done.

Now suppose $l(c)$ is not divisible by $l(a)$. Then by Lemma~\ref{lemma_abcd} we have
$c=m_1z_1$ where $m_1$ is a monomial in $a,b$ and $z=z_1z_2$ with $z=a$ or $b$ and $0<l(z_1)<l(a)$. Hence $y=qc=qm_1z_1$. Again by Lemma~\ref{lemma_abcd} we have
$b=b_1z_1$ for some $b_1$. Decompositions $z=z_1z_2$, $b=b_1z_1$ contradict to goodness of $\{a,b\}$.

The equalities of lengths are trivial.
\end{proof}

\begin{example}
\label{example_chain}
Let $A=\k\{x_1,\ldots,x_N\}$ be the free algebra with $N\ge 2$. Denote $a_1=x_1,b_1=x_2$ and define inductively $a_k,b_k$ by the rule
$$a_{k+1}=a_kb_ka_k^8b_k^8,\quad b_{k+1}=a_k^3b_k^3a_k^6b_k^6.$$
By Lemma~\ref{lemma_1423} we get that all sets $\{a_k,b_k\}$ are good and $l(a_k)=l(b_k)$ for all $k$.
Using  Theorem~\ref{theorem_xprod}, we see that $M_{a_{k+1}},M_{b_{k+1}}\in \langle M_{a_{k}},M_{b_{k}}\rangle$ while $M_{a_{k}},M_{b_{k}}\notin \langle M_{a_{k+1}},M_{b_{k+1}}\rangle$ for any $k\ge 1$. Thus we get an infinite chain of thick triangulated subcategories
$$\Perf A\supset \langle M_{a_1},M_{b_1}\rangle \varsupsetneq
\langle M_{a_2},M_{b_2}\rangle \varsupsetneq \langle M_{a_3},M_{b_3}\rangle\varsupsetneq\ldots$$
\end{example}

\begin{proof}[Proof of Proposition~\ref{prop_lattice2}] We may identify the categories $D^b(R\mmod)$ and $(\Perf A)^{\op}$ via the equivalence $K'\colon D^b(R\mmod)\to (\Perf A)^{\op}$.
For (1), consider the descending chain from Example~\ref{example_chain}. For (2), note that $\langle M_{x_1}\rangle$ and $\langle M_{x_2}\rangle$ are minimal elements  of $\LL(R_N)$ by Proposition~\ref{prop_Mxminimal} and  $\langle M_{x_1}\rangle \ne  \langle M_{x_2}\rangle$ by Theorem~\ref{theorem_xprod}. Therefore $\LL(R_N)$ has no least element. Category $\Perf R_N$ is minimal by Theorem~\ref{theorem_Hopkins} and almost maximal by  Theorem~\ref{theorem_perfbadR}. For (3)  take $T_1=\langle M_{x_1}\rangle$, $T_2=\langle M_{x_2}\rangle$, see Example~\ref{example_chain}.
\end{proof}

We can strengthen Example~\ref{example_chain} by constructing a descending binary tree of embedded subcategories.

\begin{definition}
\label{def_tree}
Let $T$ be an additive category. By a \emph{binary tree} of full subcategories in $T$ we mean a family of nonzero full subcategories $T_\e\subset T$ indexed by nodes of an infinite descending binary tree such that
\begin{itemize}
\item $T_\e\subset T_\delta$ if and only if $\e$ is a successor of $\delta$,
\item $T_\e\cap T_\delta=0$ if $\e$ and $\delta$ are  not successors of one another.
\end{itemize}
\end{definition}


\begin{lemma}
\label{lemma_XY}
Let $A$ be a free graded algebra, let $\XX,\YY\subset A^*$ be subsets such that $\XX\cap \YY=\emptyset$ and $\XX\cup \YY$ is good. Assume $M,N$ are graded $A$-modules, $M$ has an $\XX$-filtration and $N$ has a $\YY$-filtration. Then $\Hom_{\grmod A}(M,N)=0$.
\end{lemma}
\begin{proof}
The proof is by induction in the total length of filtrations of $M$ and $N$. The base is by Lemma~\ref{lemma_AxAAxA}:
$\Hom_{\grmod A}(A/xA[i],A/yA[j])=0$ for any $x\in \XX,y\in\YY$ and $i,j\in\Z$. The induction step
follows from long exact sequences of $\Ext_{\grmod A}$ groups.
\end{proof}

Let $A$ be a free graded algebra.
For any ordered pair $P=(a,b)\in A^*\times A^*$ denote
$$P^+=(a^+,b^+):=(aba^8b^8,a^2b^2a^7b^7),\qquad P^-=(a^-,b^-):=(a^3b^3a^6b^6,a^4b^4a^5b^5).$$
For any pair  of elements in $A^*$ denote
$$T_{(a,b)}:=\langle M_a,M_b\rangle\subset \Perf A.$$

\begin{lemma}
\label{lemma_zero}
Let $A$ be a free algebra.
For any good pair $ P=(a,b)$ of homogeneous elements with $l(a)=l(b)>0$ we have
\begin{equation}
\label{eq_zero}
T_{ P}\supset T_{ P^+},T_{ P^-}, \quad T_{ P^+}\cap T_{ P^-}=0.
\end{equation}
\end{lemma}
\begin{proof}
By Lemma~\ref{lemma_1423} the sets $ P^+, P^-$ are good.
By Theorem~\ref{theorem_xprod} we have
$$M_{a^+},M_{b^+},M_{a^-},M_{b^-}\in \langle M_a,M_b\rangle$$
hence the inclusion in \eqref{eq_zero} holds.
Assume now that $M\in T_{ P^+}\cap T_{ P^-}$ for some $M\in\Perf A$. By Proposition~\ref{prop_module}, we can assume that $M$ is a graded $A$-module.
By Lemma~\ref{lemma_Gdiag}, $M$ has both $\{a^+,b^+\}$ and $\{a^-,b^-\}$-filtrations (see Definition~\ref{def_filt}).  It follows now from Lemma~\ref{lemma_XY} that $M=0$, since the set $\{a^+,b^+,a^-,b^-\}$ is good by Lemma~\ref{lemma_1423}.
\end{proof}

\begin{theorem}
\label{theorem_tree}
Let $A=\k\{x_1,\ldots,x_N\}$ be a free algebra with $N\ge 2$. Then there exists an infinite descending binary tree of thick triangulated finitely generated subcategories in $\Perf A$.
\end{theorem}
\begin{proof}
Starting  with $ P=(x_1,x_2)$,  we iterate the construction from Lemma~\ref{lemma_zero}.
For any sequence
$\e=(\e_1,\ldots,\e_n)$ where $\e_i\in\{+,-\}$ we define
$$ P^{\e_1\ldots\e_n}:=( P^{\e_1\ldots\e_{n-1}})^{\e_n}.$$
Applying  Lemma~\ref{lemma_1423}, we see that for any $\e$ the pair $ P^{\e}$ is good.
By Lemma~\ref{lemma_zero},  the categories $T_{ P^{\e}}$ organize into a binary tree in the sense of Definition~\ref{def_tree}: $T_{P^\e}\subset T_{P^\delta}$ if and only if $\e$ starts with $\delta$, and the categories from different branches of the tree intersect by zero.
$$
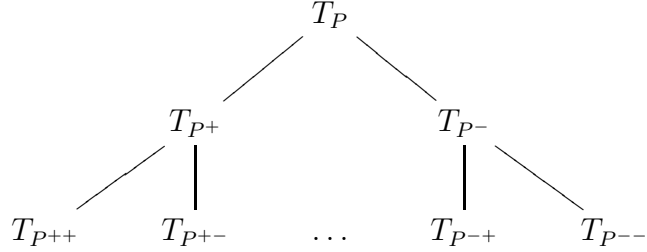
\xymatrix{&& T_{ P} \ar@{-}[rd]\ar@{-}[ld]&&\\
& T_{ P^+} \ar@{-}[ld]\ar@{-}[d]&& T_{ P^-} \ar@{-}[rd]\ar@{-}[d]\\
 T_{ P^{++}} & T_{ P^{+-}} &\ldots & T_{ P^{-+}} & T_{ P^{--}}
}$$
\end{proof}

\section{Localizations of a free algebra and subcategories in $\Perf A$}
\label{section_localization}

\subsection{Localizations of noncommutative rings}

We recall here the definition and the construction of localization of a noncommutative graded ring. We refer to \cite{Co} or to \cite{CY} for the graded case.

Let $\Lambda$ be a graded ring.
Let $g\colon F_1\to F_0$ be in $\Mat*(\Lambda)$ and $\phi\colon \Lambda\to \Omega$ be a homomorphism of graded rings. We say that
$\phi$ \emph{inverts} $g$ if the scalar extension $F_1\otimes_{\Lambda} \Omega\xra{g\otimes 1}F_0\otimes_{\Lambda} \Omega$ is an isomorphism. We say that
$\phi$ \emph{inverts} a family $\SS\subset \Mat*(\Lambda)$ if $\phi$ inverts all $g\in \SS$.

By definition, a \emph{graded localization} (or a \emph{graded universal localization}) of $\Lambda$ by $\SS$ is a graded ring $\Lambda [\SS^{-1}]$ together with a homomorphism of graded rings $\lambda_\SS\colon \Lambda \to \Lambda [\SS^{-1}]$ inverting $\SS$ such that any graded ring homomorphism $\phi\colon \Lambda \to \Omega$ inverting $\SS$ factors uniquely through $\lambda_\SS$:
$$\xymatrix{\Lambda \ar[rr]^{\lambda_\SS} \ar[rd]_{\phi} && \Lambda [\SS^{-1}] \ar@{-->}[ld] \\ & \Omega & }$$

By definition, localization is unique up to an isomorphism. Localization can be constructed explicitly as follows.
For any $g\colon F_1\to F_0$ in $\SS$ choose  bases of free graded modules $F_1, F_0$. Then
$F_1=\oplus_{j=1}^n \Lambda[c_j], F_0=\oplus_{i=1}^m \Lambda[d_i]$ and $g$ is given by a matrix $G\in Mat_{m\times n}(\Lambda)$ where $G_{ij}\in \Lambda_{d_i-c_j}$.
Consider the $n\times m$ matrix $t^g$ whose components are homogeneous formal variables $t^g_{ij}$,
$i=1\ldots n, j=1\ldots m$ with $\deg t^g_{ij}=-\deg G_{ji}=c_i-d_j$.  Consider the  graded ring $\Lambda\{t^g_{ij}\}_{g\in \SS}$ freely generated over $\Lambda$ by all variables $t^g_{ij}$. Take the quotient  of $\Lambda\{t^g_{ij}\}_{g\in \SS}$ modulo all relations of the form
$G\cdot t^g=1, t^g\cdot G=1$ (where $1$ denotes the square identity matrix of appropriate size and the given matrix relations read  as a series of relations on entries). Then this quotient is the localization $\Lambda [\SS ^{-1}]$ of $\Lambda$ by $\SS$.

The localization $\Lambda[\SS^{-1}]$ can be the zero ring; it is non-zero if and only if there exists at least one homomorphism $\Lambda\to \Omega$ to a non-zero ring $\Omega$, inverting $\SS$.

As a special case, one can consider a homomorphism of free graded modules of rank one. Such a homomorphism $g\colon \Lambda[c]\to \Lambda[d]$ is given by multiplication  by an element $a\in \Lambda_{d-c}$. A homomorphism $\phi\colon \Lambda\to\Omega$ of graded rings inverts $g$ if and only if $\phi(a)$ is invertible in~$\Omega$. Thus we get localizations of graded rings over families of homogeneous elements.

\medskip
In the commutative world only square matrices can be inverted; a square matrix is inverted iff its determinant is inverted; an element can be inverted (in a nonzero ring) iff it is not nilpotent. In the noncommutative world the questions are much more complicated. First, non-square matrices can be invertible:

\begin{example}
Let $A=\k\{x,y\}$ be the free algebra in two variables. Let $G=(x\,\,y)$ be a $1\times 2$ matrix over $A$. Then the localization $A[G^{-1}]$ is the quotient algebra
$$A\{t_1,t_2\}/(xt_1+yt_2-1, t_1x-1, t_1y, t_2x, t_2y-1).$$
Let $C=\End_{\k}(V)$ where $V$ is an infinite-dimensional $\k$-vector space. There is an isomorphism $V\cong V\oplus V$. Therefore, there is  an isomorphism 
$$C=\Hom(V,V)\cong \Hom(V,V\oplus V)=\Hom(V,V)\oplus \Hom(V,V)= C\oplus C,$$ 
which is an isomorphism of right $C$-modules. It follows that  there exists an invertible $1\times 2$ matrix over $C$. Consequently (as $A$ is free), there exists a homomorphism $A\to C$ sending $G$ to an invertible matrix. Thus $A[G^{-1}]$ is a non-zero ring.
\end{example}

The localization of a ring without zero divisors can have zero divisors as the following example demonstrates.
\begin{example}
Let $A=\k\{x,y\}$ be the free algebra.
Let
$$e:=y(xy)^{-1}x\in A[(xy)^{-1}].$$

Then $e\ne 1$ in $A[(xy)^{-1}]$. Otherwise we have $x (y(xy)^{-1})=1$ and $(y(xy)^{-1})x=1$, hence $x$ is invertible in $A[(xy)^{-1}]$, what contradicts  Theorem~\ref{theorem_invert}. Thus $e\ne 1$.

But $x(e-1)=x(y(xy)^{-1}x-1)=x-x=0$ and similarly $(e-1)y=0$. Hence $x,y$ and $e-1$ are zero divisors in $A[(xy)^{-1}]$.

Moreover, $e^2=e$ in $A[(xy)^{-1}]$.
\end{example}

The localization functor in general rings is not  exact, in contrast to commutative case.

\begin{lemma}
\label{lemma_g1g2g3}
Let $a_1,a_2,a_3\in \Lambda $ be homogeneous elements of a graded ring $\Lambda $. Then one has isomorphisms of graded rings
\begin{align*}
\Lambda [(a_1a_2a_1)^{-1}]&\cong \Lambda [a_1^{-1},a_2^{-1}];\\
\Lambda [(a_1a_2)^{-1},(a_2a_3)^{-1}]&\cong \Lambda [a_1^{-1},a_2^{-1},a_3^{-1}].
\end{align*}
\end{lemma}
\begin{proof}
This follows from the universal property of graded localizations and from the following observations: (1) an element $a_1a_2a_1$ in a ring is invertible if and only if both $a_1$ and $a_2$ are invertible; (2) elements $a_1a_2,a_2a_3$ are invertible if and only if all $a_1$, $a_2$ and $a_3$ are invertible.
\end{proof}

\subsection{Subcategories as kernels of scalar extensions}

Let $A$ be a free graded algebra. Let $T=\langle M_g\rangle _{g\in \SS}$ be the thick triangulated subcategory of $\Perf A$ classically generated by a family of dg modules $M_g$ for some set of homomorphisms $\SS\subset \Mat*(A)$. Denote by $\overline{T}\subset D(A)$ the localizing subcategory generated by the same family of dg modules $M_g, g\in\SS$ (i.e. $\overline{T}$ is the minimal triangulated subcategory of $D(A)$ which is closed under arbitrary direct sums and contains each object $M_g$, $g\in \SS$). By Theorem 2.1 in \cite{Ne2} the Verdier localization functor
$$q\colon D(A)\to D(A)/\overline{T}$$
maps $\Perf A$ to the subcategory $(D(A)/\overline{T})^c$ of compact objects and induces the full and faithful functor
$$e\colon (\Perf A)/T\to (D(A)/\overline{T})^c$$
which is the idempotent completion.

On the other hand the set $\SS\subset \Mat*(A)$ gives rise to the universal localization ring homomorphism
$$\lambda _{\SS}\colon A\to A[\SS ^{-1}]$$
and hence induces the functor of extension of scalars
$$\lambda _{\SS }^*\colon D(A)\to D(A[\SS ^{-1}]),\quad \Perf A\to \Perf
A[\SS ^{-1}]$$

The next proposition and its proof are essentially taken from \cite{NR}.

\begin{prop}
\label{equality of localizations} In the above setup there exists a triangulated functor $\overline{t}\colon D(A)/\overline{T}\to D(A[\SS ^{-1}])$ such that the diagram of functors
$$\xymatrix{D(A)\ar[rrrr]^{\lambda_\SS^*} \ar[rrd]_q &&&& D( A[\SS^{-1}]) \\
&&D(A)/\overline{T}\ar[rru]_{\overline{t}}&&}$$
commutes. Moreover, $\overline{t}$ is an equivalence.

This induces a commutative diagram of functors
$$\xymatrix{\Perf A\ar[rrrr]^{\lambda_\SS^*} \ar[rrd]_q &&&& \Perf A[\SS^{-1}], \\
&&(\Perf A)/T\ar[rru]_{t}&&}$$
where $t$ is the idempotent completion. In particular, $T=\ker \lambda ^*_{\SS}$.
\end{prop}

\begin{proof} The second part of the proposition follows from the first one in view of Theorem 2.1 in \cite{Ne2}.

To show the existence of the functor $\overline{t}$ it suffices to prove that $\overline{T}\subset \ker \lambda ^*_{\SS}$. Note that $T\subset \ker \lambda ^*_{\SS}$ by the definition of the ring $A[\SS ^{-1}]$. Also note that the functor $\lambda _{\SS}^*$ preserves arbitrary direct sums. Hence $\overline{T}\subset \ker \lambda _{\SS}^*$. This proves the first assertion. It remains to show that $\overline{t}$ is an equivalence. This will take a few steps.

First, recall the isomorphism of graded algebras
$A=\End ^\bullet _{D(A)}(A)$ which is given by $a\mapsto (l_a\colon A\to A)$, the left multiplication by $a$.
The functor $q$ induces the homomorphism of graded rings
$$\phi \colon A=\End ^\bullet _{D(A)}(A)\stackrel{q}{\to}
\End ^\bullet _{D(A)/\overline{T}}(q(A))=:B$$
This gives $B$ the structure of a graded $A$-module (via $b\cdot a:=b\circ q (l_a)$), hence we may consider $B$ as an object in $D(A)$. Note that the ring homomorphism $\phi$ above is also a morphism of graded $A$-modules.
Recall (see \cite[Lemma 1.7]{Ne2}) that the localization functor $q\colon D(A)\to D(A)/\overline{T}$ has a full and faithful right adjoint $r\colon D(A)/\overline{T}\to D(A)$ (here it is important that we have passed to unbounded derived categories). Denote by $\eta _A\colon A\to rq(A)$ the corresponding adjunction morphism in $D(A)$. By Proposition~\ref{prop_module}, dg module $rq(A)$ is quasi-isomorphic to its cohomology graded module, therefore we can assume that $rq(A)$ is just a graded $A$-module.
Moreover, 
we have
\begin{equation}
\label{eq_ArqA}
\Hom^{\bul}_{D(A)}(A,rq(A))=\Hom^{\bul}_{\grmod A}(A,rq(A))
\end{equation}
since $A$ is h-projective.

\begin{lemma} 
\label{lemma_keyNeRa} In the above notation there exists a morphism $\psi \in \Hom_{\grmod A}(B,rq(A))$ such that $\psi\circ \phi=\eta _A$. Moreover, $\psi$ is an isomorphism.
\end{lemma}

\begin{proof} 
We define $\psi$ as the composition of the maps
$$B=\Hom ^\bullet _{D(A)/\overline{T}}(q(A),q(A))\xra{\psi _1} \Hom ^\bullet _{D(A)}(A,rq(A))=\Hom ^\bullet _{\grmod A}(A,rq(A))\xra{\psi _2}  rq(A),$$
where $\psi_1$ is given by adjunction, $\psi _1(b)=r(b)\circ \eta _A$ and $\psi_2(f)=f(1)$. Clearly
$\psi_1,\psi_2$ and $\psi$ are isomorphisms of graded vector spaces. Let us prove that $\psi$ is a morphism of right $A$-modules. Because $\eta$ is a morphism of functors we have the commutative diagram in~$D(A)$
\begin{equation}
\label{eq_naturality}
\xymatrix{
A \ar[r]^{l_a} \ar[d]^{\eta_A} &  A \ar[d]^{\eta_A}\\
rq(A) \ar[r]^{rq(l_a)}  & rq(A).
}
\end{equation}
Let $b\in B$, $a\in A$. We have
\begin{align*}
\psi (b\cdot a) & =  (r(b\cdot a)\circ \eta _A)(1)\quad \text{by definition of $\psi$}\\
          & =  (r(b\circ q(l_a))\circ \eta _A)(1)\quad \text{by definition of $b\cdot a$}\\
          & =  (r(b)\circ rq(l_a)\circ \eta _A)(1)\quad \text{since $r$ is a functor}\\
          & =  (r(b)\circ \eta _A\circ l_a)(1)\quad \text{by \eqref{eq_naturality}}\\
          & =  (r(b)\circ \eta _A)(l_a(1))\quad\text{since $r(b)\circ \eta _A$ and $l_a$ are homomorphisms  of modules by \eqref{eq_ArqA}}\\
          & =  (r(b)\circ \eta _A)(1\cdot a)\\
          & =  ((r(b)\circ \eta _A)(1))\cdot a\quad \text{as $r(b)\circ \eta_A$ is a homomorphism of $A$-modules}\\
          & =  \psi (b)\cdot a\quad\text{by putting $a=1$ in the above equalities.}
\end{align*}

It remains to show that $\psi\circ \phi=\eta _A$. Because all the maps $\phi, \psi, \eta_A$ are morphisms of right $A$-modules it suffices to note that
$$\psi(\phi(1))=\psi(1_{q(A)})=(r(1_{q(A)})\circ \eta _A)(1)=\eta _A(1).$$
This proves Lemma \ref{lemma_keyNeRa}.
\end{proof}

The functor $\overline{t}$ maps $q(A)\in D(A)/\overline{T}$ to $A[\SS ^{-1}]\in D(A[\SS ^{-1}])$. This gives the graded ring homomoprhism
$$\mu \colon B=\End ^\bullet (q(A))\to \End ^\bullet (A[\SS ^{-1}])=A[\SS ^{-1}]$$
such that the diagram
\begin{equation}\label{diag comm}
\xymatrix{A\ar[rrrr]^{\lambda_\SS} \ar[rrd]_{\phi} &&&& A[\SS^{-1}].\\ && B \ar[rru]_{\mu }&&}
\end{equation}
commutes.

\begin{lemma}
\label{lemma_universal} 
The homomorphism $\mu$ is an isomorphism.
\end{lemma}

\begin{proof} It suffices to check that the homomorphism $\phi \colon A\to B$ is the universal localization with respect to the set $\SS$. That is given a homomorphism of graded ring $f\colon A\to C$ which inverts the set $\SS$ we need to show that there exists a unique ring homomorphism $f'\colon B\to C$ such that $f'\cdot \phi =f$.

Since $f$ inverts $\SS$ the induced functor $f^*\colon D(A)\to D(C)$ annihilates all modules $M_g$ for $g\in \SS$. Hence $T\subset \ker f^*$. Also $f^*$ commutes with direct sums, so $\overline{T}\subset \ker f^*$. Therefore the functor $f^*$ factors as \begin{equation}
\label{eq_fFq}
f^*=w\cdot q
\end{equation} 
for some functor $w\colon D(A)/\overline{T}\to D(C)$. Since $w(q(A))=C$ this gives the required ring homomorphism
$$f'\colon B=\End ^\bullet (q(A))\stackrel{w}{\to} \End ^\bullet (C)=C$$
such that $f'\cdot \phi =f$. It remains to prove the uniqueness of $f'$.

We consider $B$ and $C$ as graded right $A$-modules  via the homomorphisms
$\phi \colon A\to B$ and $f\colon A\to C$ respectively. A ring homomorphism $f'\colon B\to C$ such that $f'\cdot \phi =f$ is in particular a homomorphism of right graded $A$-modules. Hence it suffices to prove that the map induced by $\phi$:
$$ \Hom ^\bullet _{D(A)}(B,C)\to \Hom ^\bullet _{D(A)}(A,C)$$
is bijective. By Lemma \ref{lemma_keyNeRa} the morphism $\phi \in \Hom _{D(A)}(A,B)$ is isomorphic to the adjunction morphism
$$\eta _A\colon A\to rq(A)$$
The cone $Cone(\eta _A)$ of the adjunction map $\eta _A$ lies in the kernel $\overline{T}$ of the localization functor~$q$. Hence by \eqref{eq_fFq} $Cone(\eta _A)$ also lies in the kernel of the functor $f^*\colon D(A)\to D(C)$. It follows by standard adjunction that 
$$\Hom ^\bullet _{D(A)}(Cone (\eta _A),C)=\Hom ^\bullet _{D(C)}(f^* Cone (\eta _A),C)=0,$$ 
and so the map induced by $\eta_A$
$$\Hom ^\bullet _{D(A)}(rq(A),C)\to \Hom ^\bullet _{D(A)}(A,C)$$
is bijective. This proves the lemma.
\end{proof}

Now we can complete the proof of Proposition~\ref{equality of localizations}. It remains to show that the functor
$$\overline{t}\colon D(A)/\overline{T}\to D(A[\SS ^{-1}])$$
is an equivalence.

First we claim that $\overline{t}$ preserves arbitrary direct sums.
Indeed, we know that $q$ and $\lambda ^*_{\SS}$ preserve direct sums ($q$ preserves direct sums by Lemma 1.5 in \cite{NeBo} since $\overline{T}$ is localizing). Also the functor $q$ is essentially surjective, and the statement follows since $\overline t\circ q\cong \lambda ^*_{\SS}$.
The categories $D(A)/\overline{T}$ and $D(A[\SS ^{-1}])$ have  compact generators $q(A)$ and $A[\SS ^{-1}]$ respectively and by Lemma \ref{lemma_universal} the functor $\overline{t}$ induces the ring isomorphism
$$\mu \colon \End^\bullet (q(A))\stackrel{\sim}{\to }\End ^\bullet (A[\SS ^{-1}])$$

We now prove that $\overline{t}$ is full and faithful. Let
$$\PP =\{X\in D(A)/\overline{T}\mid \overline{t}\colon \Hom ^\bullet (q(A),X)\to
\Hom ^\bullet (\overline{t}q(A),\overline{t}X)\quad \text{is an isomorphism}\}$$
Then $\PP$ is a full triangulated subcategory of $D(A)/\overline{T}$ which contains $q(A)$ (by Lemma~\ref{lemma_universal}). Since the objects $q(A)$ and $\overline{t}q(A)$ are compact and the functor $\overline{t}$ preserves direct sums, it follows that $\PP$ is closed under direct sums. Hence $\PP=D(A)/\overline{T}$.

Now define
$$\QQ =\{Y\in D(A)/\overline{T}\mid \overline{t}\colon \Hom ^\bullet (Y,X)\to
\Hom ^\bullet (\overline{t}Y,\overline{t}X)\quad \text{is an isomorphism for all $X\in D(A)/\overline{T}$}\}$$
This is a full triangulated subcategory of $D(A)/\overline{T}$ which contains $q(A)$ and is closed under direct sums (because $\overline{t}$ preserves direct sums). Thus $\QQ=D(A)/\overline{T}$ and so the functor $\overline{t}$ is full and faithful. Finally notice that the essential image of $\overline{t}$ contains the compact generator $A[\SS ^{-1}]$ and is closed under arbitrary direct sums. So $\overline{t}$ is essentially surjective.
This finishes the proof of Proposition \ref{equality of localizations}.
\end{proof}

\begin{remark}
Proposition~\ref{equality of localizations} can also be deduced from \cite[Prop. 3.5]{CY}. 
In the setup of this Section, let $T^\perp\subset D(A)$ be the right orthogonal subcategory. Then \cite[Prop. 3.5]{CY} states that the scalar restriction functor $D(A[\SS^{-1}])\to D(A)$ restricts to an equivalence $D(A[\SS^{-1}])\xra\sim T^\perp$.
Since $T\subset D(A)$ is a localizing compactly generated subcategory, we get a semi-orthogonal decomposition $D(A)=\langle T^\perp,T\rangle$ and our Proposition~\ref{equality of localizations} follows.

Note that the scalar restriction functor $D(A[\SS^{-1}])\to D(A)$ does not preserve compact objects and thus \cite[Prop. 3.5]{CY} has no analogs for the categories $\Perf$. 
\end{remark}

\medskip
\begin{definition} Let $\phi \colon \Lambda \to \Omega$ be a homomorphism of graded $\k$-algebras. Consider the extension of scalars functor
$\phi ^*\colon \Perf \Lambda \to \Perf \Omega$ and put
$T_\phi :=\ker \phi ^*$. It is a thick triangulated subcategory
in $\Perf \Lambda$.
\end{definition}

\begin{corollary}
\label{cor_tphi}
Let $A$ be a free graded algebra. Then any full thick triangulated subcategory in $\Perf A$ is of the form $T_{\phi}$ for some graded algebra homomorphism $\phi\colon A\to B$.
\end{corollary}

\begin{proof} By  Corollary~\ref{cor_module} every object in $T$ is isomorphic to $M_g$ for some $g\in \Mat*(A)$ (Definition \ref{m-g}). Let
$$\SS =\{g\in \Mat*(A)\mid M_g\in T\}$$
Then $T=\langle M_g\rangle_{g\in \SS}$ (clear) and by Proposition~\ref{equality of localizations}
$T=T_{\lambda_\SS}$ for the localization homomorphism $\lambda_\SS\colon A\to A[\SS^{-1}]$.
\end{proof}

\begin{remark}
Note that $\phi$ is not uniquely defined by $T_{\phi}$.
\end{remark}

\begin{corollary}
\label{cor_G1G2}
Let $A$ be a free graded algebra, and let $\SS\subset \Mat*(A)$ be a family of homomorphisms. Then for any  $x\in \Mat*(A)$ we have  $M_{x}\in\langle M_g\rangle_{g\in \SS}$ if and only if $x$ is invertible over $A[\SS^{-1}]$.
\end{corollary}
\begin{proof} By Proposition~\ref{equality of localizations}, $M_{x}\in\langle M_g\rangle_{g\in \SS}$ if and only if
the scalar extension $M_x\otimes_AA[\SS^{-1}]$ is a zero object of $\Perf A[\SS^{-1}]$ which happens if and only if $x$ is invertible over $A[\SS^{-1}]$.
\end{proof}

\begin{corollary}
\label{cor_Mgall}
Let $A$ be a free graded algebra, let $\SS\subset \Mat*(A)$ be a family of homomorphisms. Then $\langle M_{g}\rangle_{g\in  \SS}=\Perf A$ if and only if the localization $A[\SS^{-1}]$ is the zero ring.
\end{corollary}
\begin{proof}
One has
\begin{multline*}
\langle M_{g}\rangle_{g\in  \SS}=\Perf A \Longleftrightarrow
A\in \langle M_{g}\rangle_{g\in  \SS}\Longleftrightarrow
A\in \ker \lambda_\SS^* \Longleftrightarrow
A[\SS^{-1}]=0\in\Perf A[\SS^{-1}] \Longleftrightarrow\\
\Longleftrightarrow A[\SS^{-1}]\quad\text{is a zero ring},
\end{multline*}
where $\lambda_\SS^*\colon \Perf A\to\Perf A[\SS^{-1}]$ denotes the scalar extension functor and the second equivalence is by Proposition~\ref{equality of localizations}.
\end{proof}

\section{Some applications}
\label{section_applications}
Now we get some applications to pure algebra. We do not know if these results can be deduced directly.

\begin{theorem}
\label{theorem_invert}
Let $A$ be a free graded algebra. Let $\XX\subset A^*$ be a subset. Let $\~\XX\subset A^*$ be the family provided by Proposition~\ref{prop_makegood}, in particular, one can take $\~\XX=\XX$ if $\XX$ is good. Then a homogeneous element $x\in A$ is invertible in
$A[\XX^{-1}]$ if and only if
$$x=\lambda\cdot \prod_{i=1}^n y_i\quad\text{for some}\quad \lambda\in\k^*, \quad y_i\in\~\XX \quad \text{and}\quad n\ge 0.$$
\end{theorem}
\begin{proof}
Follows from Corollary~\ref{cor_G1G2} and Theorem~\ref{theorem_xprod2}.
\end{proof}

\begin{corollary}
\label{cor_locinjective}
Let $A$ be a free graded algebra. Let $\XX\subset A$ be a family of nonzero homogeneous elements. Then the localization $A[\XX^{-1}]$ is not a zero ring, moreover, the canonical map
$$A\to A[\XX^{-1}]$$
is injective.
\end{corollary}
\begin{proof}
It suffices to check the statements for $\XX=A^*$ being the set of all nonzero homogeneous elements in $A$. Indeed, the localization $A\to A[(A^*)^{-1}]$ factors through any other localization  $A\to A[\XX^{-1}]$. So we assume $\XX=A^*$.

By Corollary~\ref{cor_Mgall}, if the ring $A[(A^*)^{-1}]$ is zero  then the modules
$A/xA$, $x\in A^*$, generate the category $\Perf A$. This contradicts Corollary~\ref{cor_Mgproper}.

Thus the ring $A[(A^*)^{-1}]$ is nonzero. Consequently, the localization homomorphism $f\colon A\to A[(A^*)^{-1}]$ is injective: $f(a)$ is invertible in $A[(A^*)^{-1}]$ for any $a\in A^*$, hence $f(a)\ne 0$.
\end{proof}

\begin{example}
Let $A=\k\{x,y,\ldots\}$ be a free graded algebra.
\begin{enumerate}
\item Let $g=x$ or $xy$. Then the only elements in $A$ that are invertible in $A[g^{-1}]$ are powers of $g$ (up to scalars).
\item Let $g=xyx$. Then all monomials in $x$ and $y$ are invertible in $A[g^{-1}]$.
\end{enumerate}
\end{example}

\section{Support}


In this final section we introduce a notion of support for an object in $\Perf A$, which sometimes helps to distinguish between different subcategories in $\Perf A$.

Let $A^{ab}$ be the quotient of $A$ modulo its commutator ideal, clearly $A^{ab}$ is the graded polynomial algebra $\k[x_1,\ldots,x_N]$.  Consider $\Proj A^{ab}\cong \P^{N-1}$ ---
the set of prime homogeneous ideals in $A^{ab}$ different from the augmentation ideal.  $\Proj A^{ab}$ is equipped
with Zariski topology. For a graded $A^{ab}$-module $N$ its support is defined as the set of all $\p\in \Proj A^{ab}$ such that the homogeneous localization $N_{\p}$ is nonzero. Then the support of a finitely generated module is closed in Zariski topology.

For an object $M\in \Perf A$ consider its abelianization $M^{ab}=M\otimes^L_AA^{ab}$.  Cohomology of $M^{ab}$ is a graded finitely generated $A^{ab}$-module. Define the support $\Supp M\subset \Proj A^{ab}$ as the support of $H(M^{ab})$, it is a Zariski closed set.

\begin{lemma}
The following readily holds:
\label{lemma_suppprop}
\begin{enumerate}
\item $\Supp M=\Supp M[n]$ for any $M\in\Perf A$ and $n\in\Z$;
\item for a distinguished triangle $M_1\to M_2\to M_3\to M_1[1]$ in $\Perf A$ one has
$$\Supp M_2\subset \Supp M_1\cup \Supp M_3;$$
\item for any $N\in \langle M\rangle\subset \Perf A$ one has $\Supp N\subset \Supp M$.
\end{enumerate}
\end{lemma}

\begin{lemma}
\label{lemma_suppbig}
Let $M\in \Perf A$ be a nonzero object. Then $\Supp M\ne \emptyset$  and
$$\codim \Supp M\le 1.$$ Moreover, if $M=Cone(F_1\to F_0)$ where $F_0$ and $F_1$ are free $A$-modules and $\rank F_0\ne \rank F_1$ then $\Supp M=\P^{N-1}$.
\end{lemma}
\begin{proof}
By Corollary~\ref{cor_module}, $M$ is quasi-isomorphic to $M_G=Cone(F_1\xra{G} F_0)$
for some morphism between free $A$-modules of finite rank $F_0$ and $F_1$, given by a matrix $G\in \Mat*(A)$. Then $M^{ab}=Cone(F_1^{ab}\xra{G^{ab}} F_0^{ab})$, where $F_0^{ab}, F_1^{ab}$ are free graded $\k[x_1,\ldots,x_N]$-modules. One has $H(M^{ab})=\ker (G^{ab})[1]\oplus \coker (G^{ab})$. If $G^{ab}$ is not injective then $\ker (G^{ab})$ is torsion-free and thus $\Supp M=\P^{N-1}$. If $G^{ab}$ is injective then $\pd \coker G^{ab}\le 1$ and thus  $\codim \Supp M=\codim \Supp \coker G^{ab}\le 1$ (see for example \cite[Cor. 18.5 and Th. 18.7]{Ei}).

To see that the support of $M$ is non-empty, note that the free graded modules $F_0$ and $F_1$ are not isomorphic. Hence the free graded modules $F_0^{ab}$ and $F_1^{ab}$ also are not isomorphic and  $H(M^{ab})\ne 0$.

For the last statement, note that the localization of $G^{ab}$ at the generic point $\eta\in \Proj A^{ab}$ is not an isomorphism  (as $\rank F_0^{ab}\ne \rank F_1^{ab}$). Therefore $\eta\in\Supp M$ and consequently $\Supp M=\Proj A^{ab}$.
\end{proof}

For a homogeneous polynomial $f\in \k[x_1,\ldots,x_N]$ denote by $Z(f)\subset \P^{N-1}$ the zero locus of $f$. It is a hypersurface provided that $\deg f>0$.

\begin{lemma}
\label{lemma_MgSupp}
Let $x\in A^*$, let $x^{ab}$ be its image in $A^{ab}$. Then
\begin{equation*}
\Supp M_x=
\begin{cases} \P^{N-1}, & \text{if}\  x^{ab}=0;\\
Z(x^{ab}), & \text{if}\ x^{ab}\ne 0.
\end{cases}
\end{equation*}
\end{lemma}
\begin{proof}
Clear.
\end{proof}

\begin{prop}
\label{prop_support}
Let $x,y\in A^*$.
\begin{enumerate}
\item Suppose $Z(x^{ab})\not\subset Z(y^{ab})$. Then $M_{x}\notin \langle M_{y}\rangle$.
\item Suppose $y^{ab}\ne 0$. Let $M=Cone(F_1\to F_0)$ where $F_0,F_1$ are free $A$-modules and $\rank F_0\ne \rank F_1$. Then $M\notin \langle M_{y}\rangle$.
\item Suppose $Z(x^{ab})\cap Z(y^{ab})$ has $\codim \ge 2$ in $\P^{N-1}$. Then
$$\langle M_{x}\rangle\cap \langle M_{y}\rangle=0.$$
\end{enumerate}
\end{prop}
\begin{proof}
(1) and (2) follow from Lemmas~\ref{lemma_suppprop},~\ref{lemma_MgSupp} and~\ref{lemma_suppbig}. For (3), assume $M\in \langle M_{x}\rangle\cap \langle M_{y}\rangle$ and $M\ne 0$. Then $\Supp M\subset \Supp M_{x}\cap \Supp M_{y}=Z(x^{ab})\cap Z(y^{ab})$ and thus $\Supp M$ has codimension $\ge 2$. It contradicts to Lemma~\ref{lemma_suppbig}.
\end{proof}

\begin{remark} Parts (1) and (2) of
Proposition~\ref{prop_support} can be proved without using supports. They follow from Corollary~\ref{cor_G1G2}. Part (3) does not follow from  Corollary~\ref{cor_G1G2} in general. But in some cases (like $x,y$ are good) part (3) follows from the minimality, see Proposition~\ref{prop_Mxminimal}.
\end{remark}

\end{document}